\documentclass[reqno]{amsart}
\usepackage{latexsym,amssymb}
\usepackage{amsmath}
\usepackage{amsthm} 
\usepackage[mathscr]{eucal}
\usepackage{color}
\usepackage[abbrev]{amsrefs}
\usepackage[T1]{fontenc}

\usepackage{latexsym}
\usepackage{amssymb}
\usepackage{amsfonts}
\usepackage{enumerate}
\usepackage{mathrsfs}
\usepackage{color}
\usepackage{comment}
\usepackage{mathtools}

\mathtoolsset{showonlyrefs=true}

\newtheorem{Theorem}{\bf Theorem}[section]
\newtheorem{Lemma}{\bf Lemma}[section]
\newtheorem{Proposition}{\bf Proposition}[section]
\newtheorem{Corollary}{\bf Corollary}[section]
\newtheorem{Remark}{\bf Remark}[section]
\newtheorem{Example}{\bf Example}[section]
\newtheorem{Definition}{\bf Definition}[section]

\newenvironment{theorem}{\begin{Theorem}$\!\!\!$}{\end{Theorem}}
\newenvironment{lemma}{\begin{Lemma}$\!\!\!$}{\end{Lemma}}
\newenvironment{proposition}{\begin{Proposition}$\!\!\!$}{\end{Proposition}}

\newenvironment{remark}{\begin{Remark}$\!\!\!$}{\end{Remark}}

\newenvironment{definition}{\begin{Definition}$\!\!\!$}{\end{Definition}}

\numberwithin{equation}{section}

\def\XXint#1#2#3{{\setbox0=\hbox{$#1{#2#3}{\int}$}
		\vcenter{\hbox{$#2#3$}}\kern-.5\wd0}}

\begin{document}
	
	\title[time-fractional semilinear heat equation]
	{On solvability of a time-fractional semilinear heat equation, and its quantitative approach to the classical counterpart}
	
	%    Remove any unused author tags.
	
	%    author one information
	\author{Kotaro Hisa}
	\address[Kotaro Hisa]{Graduate School of Mathematical Sciences, The University of Tokyo, 3-8-1 Komaba,
Meguro-ku, Tokyo 153-8914, Japan.}
	%\curraddr{}
	\email{hisak@ms.u-tokyo.ac.jp}
	%\thanks{}

	\author{Mizuki Kojima}
	\address[Mizuki Kojima]{Department of Mathematics, Tokyo Institute of Technology,
		2-12-1 Ookayama, Meguro-ku, Tokyo 152-8551, Japan. }
	%\curraddr{}
	\email{kojima.m.aq@m.titech.ac.jp}
	%\thanks{$^\ast$Corresponding author: Mizuki Kojima.}
	\thanks{The second author was supported by JST, the establishment of university fellowships towards the creation of science technology innovation, Grant Number JPMJFS2112.}

	
	
	%\date{}
	
	%\dedicatory{}
	
	%%RUNNING HEAD
	%\pagestyle{myheadings}
	%\markboth{M. Kojima and K. Hisa}{nonnegative solutions of a time-fractional semilinear heat equation}
	
	\begin{abstract}
		We are concerned with the following time-fractional semilinear parabolic problem in the $N$-dimensional whole space ${\bf R}^N$ with $N \geq 1$,
		\[
		{\rm (P)}_\alpha \qquad 
          {}^{C}\partial_t^\alpha u -\Delta u = u^p,\quad  t>0,\,\,\, x\in{\bf R}^N, \qquad
		    u(0) = \mu  \quad \mbox{in}\quad {\bf R}^N,
		\]
		where ${}^{C}\partial_t^\alpha$ denotes the Caputo derivative of order $\alpha \in (0,1)$, $p>1$, and $\mu$ is a nonnegative Radon measure on ${\bf R}^N$.  The case of $\alpha=1$ formally gives the Fujita-type equation (P)$_1$ \ $\partial_tu-\Delta u=u^p$. In particular, we mainly focus on the Fujita critical case $p=p_F:=1+2/N$. It is well known that the Fujita exponent $p_F$ separates the ranges of $p$ for the global-in-time solvability of (P)$_1$. In particular, (P)$_1$ with $p=p_F$ possesses no global-in-time solutions, and is not locally-in-time solvable in its scale critical space $L^1(\mathbf{R}^N)$. It is also known that the exponent $p_F$ plays the same role for the global-in-time solvability for (P)$_\alpha$. However, the problem (P)$_\alpha$ with $p=p_F$ is globally-in-time solvable, and exhibits local-in-time solvability in its scale critical space $L^1(\mathbf{R}^N)$. The purpose of this paper is to clarify the collapse of the global-in-time solvability and the local-in-time solvability of (P)$_\alpha$ as $\alpha$ approaches $1-0$.
	\end{abstract}
%%%%%%%	
	\maketitle
	%%%%%%%%%%%%%%%%%%%%%%%%%%%%%%%%%%%%%
	%%%%%%%%%%%%%%%%%%%%%%%%%%%%%%%%%%%%%
	\section{Introduction}
	%%%%%%%%%%%%%%%%%%%%%%%%%%%%%%%%%%%%%
	%%%%%%%%%%%%%%%%%%%%%%%%%%%%%%%%%%%%%
	%%%%%%%%%%%%%%%%%%%%%%%%%%%%%%%%%%%%%
	\subsection{Introduction.}
	%%%%%%%%%%%%%%%%%%%%%%%%%%%%%%%%%%%%%
	Let us consider nonnegative solutions of the following time-fractional semilinear parabolic problem in the $N$-dimensional whole space ${\bf R}^N$ with $N \geq 1$,
	\begin{equation}
	\label{TFF}
	\left\{
	\begin{split}
	{}^{C}\partial_t^\alpha u -\Delta u &= u^p\ &&t>0, \,\,\, x\in {\bf R}^N,\\
	u(0) &= \mu\  &&\mbox{in}\ \ {\bf R}^N,\\
	\end{split}
	\right.
	\end{equation}
	where $p>1$  and
	$\mu$ is a nonnegative Radon measure on ${\bf R}^N$. 
	Here, ${}^{C}\partial_t^\alpha$ is the Caputo derivative of order $\alpha\in (0,1)$, that is
	\[
		({}^{C}\partial_t^\alpha f)(t) := \frac{1}{\Gamma(1-\alpha)} \frac{d}{dt} \int_0^t \frac{f(s)-f(0)}{(t-s)^{\alpha}}\,ds,
	\] 
	where $\Gamma$ is the gamma function. This mathematical model 
      has been proposed to interpret anomalous diffusion, which is different from the conventional one based on the Brownian motion. For the mathematical treatment of time-fractional problems, see, e.g., \cite{Podl99, GalWar20, LuchYama19}.
	
	The case of $\alpha=1$ corresponds to the Fujita-type problem
	\begin{equation}
	\label{Fjt}
	\partial_t  u -\Delta u = u^p,\quad  t>0,\,\,\, x\in{\bf R}^N, \qquad
	u(0) = \mu  \quad \mbox{in}\quad {\bf R}^N,
	\end{equation}
	which has been studied extensively by many mathematicians since the pioneering work of Fujita \cite{Fujita66}. (See, e.g., \cite{QuitSoup19}, which includes an extensive list of references for problem \eqref{Fjt}.) In what follows, denote the Fujita exponent by $p_F:=1+2/N$. For the global-in-time solvability of problem \eqref{Fjt}, it is known that the condition $1<p\le p_F$ implies the nonexistence of any positive global-in-time solutions, while $p>p_F$ guarantees the global-in-time solvability of problem \eqref{Fjt} for appropriate initial data.  We emphasize that in the Fujita critical case, i.e., when $p=p_F$, the global-in-time solvability fails. See \cite{Fujita66, Sugitani75, Hayakawa73, KoSiTa77}. Furthermore, the local-in-time solvability of problem \eqref{Fjt} in its scale critical space $L^{{N(p-1)/2}}(\mathbf{R}^N)$ is guaranteed, provided $N(p-1)/2>1$. However, as for the doubly critical case where $p=p_F$ and $\mu\in L^1({\bf R}^N)$ (i.e., the simultaneous occurrence of the Fujita critical case and the scale critical case), it is known that problem \eqref{Fjt} may not have any positive local-in-time solutions for certain singular initial data. See e.g. \cite{BreCaz96, Miya21, BarasPierre85,HisaIshige18}. Roughly speaking, problem \eqref{Fjt} with $p=p_F$ exhibits the global-in-time unsolvability and local-in-time unsolvability.
	
	On the other hand, it is shown in Zhang and Sun \cite{ZhanSun15} that also for problem \eqref{TFF}, the common exponent $p_F$ becomes the critical one which separates the range of $p$ for the nonexistence and the existence of global-in-time solutions, that is, problem \eqref{TFF} is not globally-in-time solvable if $1<p<p_F$, 
    while problem \eqref{TFF} possesses a global-in-time solution if $p\ge p_F$. Note that a global-in-time solution can exist even if $ p=p_F$. Moreover, for the local-in-time solvability, \cite{GMS22, ZLS19} showed that problem \eqref{TFF} is locally-in-time solvable for the doubly critical case, i.e., $p=p_F$ and $\mu \in L^1({\bf R}^N)$. For the time-fractional semilinear heat equation, see e.g. \cite{GMS22, GalWar20, ZhanSun15, ZLS19, OkaZhan23}. Moreover, for the space-time fractional problems, see \cite{CarFerNoe24} and references therein.
  
    From the above, we are led to face the critical transition in the nature of problem \eqref{TFF} as $\alpha$ approaches $1-0$. In this paper, we shall call these differences, the \textit{collapse} of the solvability of \eqref{TFF} when $\alpha\to 1-0$.

	\begin{table}[h]\label{tab. comparison}
		\caption{Comparison of the two problems with $p=p_F$.}
		\centering
		\begin{tabular}{ccc}
			\hline
			Problem & Global solution & Local solution for $\mu\in L^{1}({\bf R}^N)$ \rule[0mm]{0mm}{5mm}\\
			\hline \hline
			\eqref{Fjt} with $p=p_F$ &not exist& 
                no sol. for some $\mu\in L^1(\mathbf{R}^N)$
            \rule[0mm]{0mm}{5mm}\\
			\eqref{TFF} with $p=p_F$ &exists& exists for all $ \mu\in L^1(\mathbf{R}^N)$
            \rule[0mm]{0mm}{5mm}\\  
            \hline
		\end{tabular}
	\end{table}
	
	Nevertheless, since the Caputo derivative is believed to be a naturally generalized notion of the classical one, it would be reasonable to expect that problems \eqref{TFF} and \eqref{Fjt} should be correlated in a consistent way when $\alpha\to 1-0$. 

    Hence, our main concern here is stated as follows.
	\vskip\baselineskip
	
	\noindent\textbf{Question.} \textit{ Can one understand in a rigorous way the mechanism of
      the collapse of the global-in-time solvability and the local-in-time solvability of problem \eqref{TFF} 
        when $\alpha$ approaches $1-0$?}

	\vskip\baselineskip
	
	To give an affirmative answer to this question, in the first part of this paper, 
      we give a necessary condition for the solvability of \eqref{TFF} which is an analogous version of the results of \cite{HisaIshige18} generalized to the time-fractional problem \eqref{TFF}. 
         Furthermore, sufficient conditions for the solvability of \eqref{TFF} are derived in a way that guarantees the optimality of necessary condition.
         These optimal conditions enable us to obtain sharp estimates of the life span of solutions of problem \eqref{TFF}, and through these estimates, we aim to investigate the collapse of the solvability, which is the final goal of this paper.
          For the exact definition of the life span, see Section 4 below.
	The second part of this paper is devoted to the application of these conditions to 
      the analysis of the collapse phenomena. 
	    Through this procedure, the occurrence mechanism of the collapse of 
	      the global-in-time solvability and the local-in-time solvability of problem \eqref{TFF} are clarified.

	%%%%%%%%%%%%%%%%%%%%%%%%%%%%%%%%%%%%%
	\subsection{Notation and the definition of solutions.}
	%%%%%%%%%%%%%%%%%%%%%%%%%%%%%%%%%%%%%
	
	In order to state our main results, we fix some notation and formulate the definition of solutions.       
        Let $B(z;r):=\{y\in{\bf R}^N; |x-y|<r\}$ for $z\in{\bf R}^N$ and $r>0$. 
           In what follows, for a given Banach space $X$, 
             we denote its norm by $\| \cdot | X\|$.    
        For $t>0$ and $x\in{\bf R}^N$, let $G$ be the fundamental solution of the heat equation in $(0,\infty)\times{\bf R}^N$, that is
	\[
	G(t,x) := \frac{1}{(4\pi t)^\frac{N}{2}} \exp\left(-\frac{|x|^2}{4t}\right).
	\]
	For $t>0$, $x\in{\bf R}^N$, and a Radon measure $\mu$ on ${\bf R}^N$, define
	\[
	e^{t\Delta} \mu(x) := \int_{{\bf R}^N} G(t, x-y) \,d\mu(y).
	\]
	We often identify a locally integrable function $\phi$ defined in ${\bf R}^N$ with the Radon measure $\phi\,dx$. Let us define
	\[
	P_{\alpha}(t)\mu := \int_0^\infty h_\alpha (\theta) e^{t^\alpha \theta \Delta}\mu \, d\theta, \quad S_{\alpha} (t)\mu := \int_0^\infty \theta h_\alpha (\theta ) e^{t^\alpha\theta \Delta} \mu \, d\theta,
	\]
	for $t>0$, $\alpha\in (0,1)$, and a Radon measure $\mu$ on $\textbf{R}^N$. Here, $h_\alpha$ is a certain probability density function on $(0,\infty)$ which satisfies $h_\alpha>0$. In addition, we have
	\begin{equation}
	\label{eq:1.4}
	\int_0^\infty \theta^{\delta} h_\alpha(\theta) \, d\theta = \frac{\Gamma(1+\delta)}{\Gamma(1+\alpha\delta)}<\infty \quad \mbox{for} \quad \delta>-1
	\end{equation}
	and
	\begin{equation}
	\label{eq:1.4.5}
	\int_0^\infty h_\alpha (\theta ) e^{z\theta} \,d\theta = E_{\alpha,1} (z), \quad 
	\int_0^\infty \alpha \theta h_\alpha (\theta ) e^{z\theta} \,d\theta = E_{\alpha,\alpha} (z), 
	\quad \mbox{for} \quad z\in{\bf  C}.
	\end{equation}
	Here, $E_{\alpha,\beta}$ is the Mittag--Leffler function, that is, for $\alpha,\beta \in{\bf  C}$ with
	${\rm Re}(\alpha)>0$ and $z\in{\bf  C}$,
	\[
	E_{\alpha,\beta}(z) := \sum_{k=0}^\infty \frac{z^k}{\Gamma(\alpha k+ \beta)}.
	\]
	For more details, see e.g. \cite{GMS22,GalWar20,Podl99,ZhanSun15,ZLS19,Mainardi1994} and references therein.
	For $a,b>0$, let $B(a,b)$ be the Beta function, that is
	\[
	B(a,b):= \int_0^1(1-s)^{a-1} s^{b-1}\,ds.
	\]

	Next, we formulate the definition of solutions of problem \eqref{TFF}. 
%%%%%%%%%%%%%%%%%%%%%%%%%%%%%%%%%%%%%%%%%%%%%%%%%%%%%%%%%%%%%%%%%%%%%%%%%%%%%%%%%%%%%%%%
%%%%%%%%   Definition of solution   %%%%%%%%%%%%%%%%%%%%%%%%%%%%%%%%%%%%%%%%%%%%%%%%%%%%
%%%%%%%%%%%%%%%%%%%%%%%%%%%%%%%%%%%%%%%%%%%%%%%%%%%%%%%%%%%%%%%%%%%%%%%%%%%%%%%%%%%%%%%%	
	\begin{definition}
		\label{Def:1.1}
		Let $T\in(0,\infty]$ and $u$ be a nonnegative measurable function in $(0,T)\times {\bf R}^N$.  For a given nonnegative Radon measure $\mu$ on ${\bf R}^N$, we say that $u$ is a solution 
		of problem \eqref{TFF} in $[0,T)\times{\bf R}^N$ if u satisfies
		\[
		\infty> u(t,x) = [P_\alpha(t)\mu] (x) +\alpha\int_0^t (t-s)^{\alpha-1} [S_\alpha(t-s)u(s)^p](x) \,ds
		\]
		for almost all $t\in(0,T)$ and $x\in{\bf R}^N$.
		If $u$ satisfies the above equality with $=$ replaced by $\ge$, then $u$ is said to be a supersolution of problem \eqref{TFF} in $[0,T)\times{\bf R}^N$.
	\end{definition}
	%
	%%%%%%%%%%%%%%%%%%%%%%%%%%%%%%%%%%%%%
	\subsection{Main results I (necessary and sufficient conditions)}
	%%%%%%%%%%%%%%%%%%%%%%%%%%%%%%%%%%%%%
	Now we are ready to state the main results of this paper. The first theorem states necessary conditions for the existence of solutions of problem \eqref{TFF}.
%%%%%%%%%%%%%%%%%%%%%%%%%%%%%%%%%%%%%%%%%%%%%%%%%%%%%%%%%%%%%%%%%%%%%%%
%%%%%    Theorem 1.1     %%%%%%%%%%%%%%%%%%%%%%%%%%%%%%%%%%%%%%%%%%%%%%
%%%%%%%%%%%%%%%%%%%%%%%%%%%%%%%%%%%%%%%%%%%%%%%%%%%%%%%%%%%%%%%%%%%%%%%
	\begin{theorem}
		\label{Thm:1.1}
		Suppose that problem \eqref{TFF}  possesses a solution in $[0,T)\times{\bf R}^N$, where $T\in (0,\infty)$.
		Then there exists a constant $\gamma_1=\gamma_1(N,p,\alpha)>0$ satisfying $\limsup_{\alpha\to1-0} \gamma_1(\alpha)\in(0,\infty)$  such that $\mu$ satisfies
		\begin{equation}
		\label{Thm:1.1.1}
		\sup_{z\in{\bf R}^N} \mu(B(z;\sigma)) \le \gamma_1 \sigma^{N-\frac{2}{p-1}}
		\end{equation}
		for all $\sigma \in (0,T^{\alpha/2})$.  In addition, if $p=p_F$, then there exists a constant $\gamma'_1=\gamma'_1(N,\alpha)>0$ satisfying $\limsup_{\alpha\to1-0} \gamma'_1(\alpha)\in(0,\infty)$  such that $\mu$ satisfies
		\begin{equation}
		\label{Thm:1.1.2}
		\sup_{z\in{\bf R}^N} \mu(B(z;\sigma)) \le \gamma'_1 \left(\int_{\sigma^{2/\alpha}/(16T)}^{1/4} t^{-\alpha} \, dt \right)^{-\frac{N}{2}}
		\end{equation}
		for all $\sigma \in (0,T^{\alpha/2})$. 
	\end{theorem}
%%%%%%%%%%%%%%%%%%%%%%%%%%%%%%%%%%%%%%%%%%%%%%%%%%%%%%%%%%%%%%%%%%%%%%%%%%%
%%%%%%%%%%    Remark 1.1      %%%%%%%%%%%%%%%%%%%%%%%%%%%%%%%%%%%%%%%%%%%%%
%%%%%%%%%%%%%%%%%%%%%%%%%%%%%%%%%%%%%%%%%%%%%%%%%%%%%%%%%%%%%%%%%%%%%%%%%%%
			We compare Theorem~\ref{Thm:1.1} with the related studies. For the solvability of problem \eqref{Fjt}, Baras and Pierre \cite{BarasPierre85} provided necessary conditions for the existence of nonnegative solutions. Subsequently, the optimality of these conditions are proved by the first author of this paper and Ishige \cite{HisaIshige18}:
			If problem \eqref{Fjt} possesses a nonnegative solution in $[0, T)\times{\bf R}^N$ for some $T>0$, then the following hold.
			\begin{enumerate}
				\item[(i)] \ There exists $C_1=C_1(N,p)>0$ such that
				\begin{equation}\label{eq. HisaIshige}
				\sup_{z\in{\bf R}^N} \mu(B(z;\sigma)) \le C_1 \sigma^{N-\frac{2}{p-1}},\quad 0<\sigma<T^\frac{1}{2}.
				\end{equation}
%%%%%%%				
				\item[(ii)] \  For the case of $p=p_F$, there exists $C_2=C_2(N)>0$ such that
				\begin{equation}\label{eq. HisaIshige Pf}
				\sup_{z\in{\bf R}^N} \mu(B(z;\sigma)) \le C_2 \left[\log\left(e+\frac{T^\frac{1}{2}}{\sigma}\right)\right]^{-\frac{N}{2}},\quad 0<\sigma<T^\frac{1}{2}.
				\end{equation}
			\end{enumerate}
			For studies on necessary conditions for the solvability of Fujita type problems, see also \cite{IshiKawaOka2020, FHIL23, HIT18, LaiSie21, Takahashi16}.

			Obviously \eqref{Thm:1.1.1} is a generalization of \eqref{eq. HisaIshige}. 
    Moreover, taking the limit $\alpha\to1-0$ in \eqref{Thm:1.1.2}, we get
			\[
			\sup_{z\in{\bf R}^N} \mu(B(z;\sigma)) \le C\left(\log\frac{4T}{\sigma^2}\right)^{-\frac{N}{2}} \le C\left[\log\left(e+\frac{T^{\frac{1}{2}}}{\sigma}\right)\right]^{-\frac{N}{2}}
			\]
			for all $\sigma \in (0,T^{1/2})$, which is the counterpart of \eqref{eq. HisaIshige Pf}. However, \eqref{Thm:1.1.2} gives not only an analogy to \eqref{eq. HisaIshige Pf} but also some information for the case of $\alpha \in (0,1)$, which differs from the case of $\alpha =1$. Namely, \eqref{Thm:1.1.2} does not exclude the case of $T=\infty$ for suitable nontrivial initial data. In contrast, \eqref{eq. HisaIshige Pf} with $T=\infty$ implies $\mu \equiv 0$.

%%%%%%%%%%%%%%%%%%%%%%%%%%%%%%%%%%%%%%%%%%%%%%%%%%%%%%%%%%%%%%%%%%%%%%%%%%%%%%%%%%%%%%%%%%%%
%%%%%%%%   Remark 1.2    %%%%%%%%%%%%%%%%%%%%%%%%%%%%%%%%%%%%%%%%%%%%%%%%%%%%%%%%%%%%%%%%%%%
%%%%%%%%%%%%%%%%%%%%%%%%%%%%%%%%%%%%%%%%%%%%%%%%%%%%%%%%%%%%%%%%%%%%%%%%%%%%%%%%%%%%%%%%%%%%	
	\begin{remark}
		{\rm 
		As mentioned above, Zhang and Sun \cite{ZhanSun15} showed that, if $1<p<p_F$ and $\mu \not \equiv0$ in ${\bf R}^N$, then problem \eqref{TFF} possesses no global-in-time solutions. We observe that Theorem~\ref{Thm:1.1} leads to the same conclusion as \cite{ZhanSun15}, since $\sigma^{N-2/(p-1)} \to 0$ as $\sigma\to\infty$.
		}
	\end{remark}
%%%%%%%%%%%%%%%%%%
	Next, we give sufficient conditions for the solvability of problem \eqref{TFF}.
%%%%%%%%%%%%%%%%%%%%%%%%%%%%%%%%%%%%%%%%%%%%%%%%%%%%%%%%%%%%%%%%%%%%%%%%%%%%%%%%%%
%%%%%%  Theorem 1.2    %%%%%%%%%%%%%%%%%%%%%%%%%%%%%%%%%%%%%%%%%%%%%%%%%%%%%%%%%%%
%%%%%%%%%%%%%%%%%%%%%%%%%%%%%%%%%%%%%%%%%%%%%%%%%%%%%%%%%%%%%%%%%%%%%%%%%%%%%%%%%%
	\begin{theorem}
		\label{Thm:1.2}
		Let $1<p\le p_F$ and $T>0$. Then there exists $\gamma_2=\gamma_2(N,p,\alpha)>0$ such that, 
          if $\mu$ is a nonnegative Radon measure on ${\bf R}^N$ satisfying
		\begin{equation}\label{Thm:1.2.1}
		\sup_{z\in{\bf R}^N} \mu(B(z;T^\frac{\alpha}{2})) \le \gamma_2  T^{\alpha(\frac{N}{2}-\frac{1}{p-1})},
		\end{equation}
		then problem \eqref{TFF} possesses a solution in $[0,T)\times{\bf R}^N$. Furthermore, the following hold.
    \begin{enumerate}
      \item[{\rm (i)}] \ If $1<p<p_F$, then it follows that
       \begin{equation}\label{est:gamma2:sub}
          \limsup_{\alpha\to1-0} \ \gamma_2(N,p,\alpha)\in (0,\infty).
       \end{equation}
%%%%%%%%%%%%%%%%%
      \item[{\rm (ii)}] \
		If $p=p_F$, then one can take $\gamma'_2(\alpha)$ such that
		\begin{equation}\label{est:gamma2:critical}
		\gamma_2(N, p_F, \alpha) \le  \gamma'_2(\alpha) (1-\alpha)^{\frac{N}{2}}
		\end{equation}
		and $\limsup_{\alpha\to1-0}\gamma'_2(\alpha)\in (0,\infty)$.
	\end{enumerate}

	\end{theorem}
%%%%%%%%%%%%%%%%%%%%%%%%%%%%%%%%%%%%%%%%%%%%%%%%%%%%%%%%%%%%%%%%%%%%%%%%%%%%%%%%%%%
%%%%%%   Theorem 1.3     %%%%%%%%%%%%%%%%%%%%%%%%%%%%%%%%%%%%%%%%%%%%%%%%%%%%%%%%%%
%%%%%%%%%%%%%%%%%%%%%%%%%%%%%%%%%%%%%%%%%%%%%%%%%%%%%%%%%%%%%%%%%%%%%%%%%%%%%%%%%%%	
	\begin{theorem}
		\label{Thm:1.3}
		Suppose that $p>p_F$ and $T>0$. Let $r>1$. Then there exists $\gamma_3=\gamma_3(N,p, r, \alpha)>0$ such that, if $\mu$ is a nonnegative measurable function in ${\bf R}^N$ satisfying
		\begin{equation}
		\label{Thm:1.3.1}
		\sup_{z\in{\bf R}^N}\left(\int_{B(z;\sigma)} \mu(y)^r \, dy\right)^\frac{1}{r}\le \gamma_3 \sigma^{\frac{N}{r}-\frac{2}{p-1}}
		\end{equation}
		for all $\sigma \in (0,T^{\alpha/2})$, then problem \eqref{TFF} possesses a solution in $[0,T)\times{\bf R}^N$. Furthermore, we have 
   \begin{equation*}
       \limsup_{\alpha\to1-0}  \  \gamma_3(N, p, r, \alpha)\in (0,\infty).
    \end{equation*}

	\end{theorem}
%%%%%%%%%%%%%%%%%%%%%%%%%%%%%%%%%%%%%%%%%%%%%%%%%%%%%%%%%%%%%%%%%%%%%%%%%%%%%%%%%
%%%%%%%%%  Remark 1.3    %%%%%%%%%%%%%%%%%%%%%%%%%%%%%%%%%%%%%%%%%%%%%%%%%%%%%%%%
%%%%%%%%%%%%%%%%%%%%%%%%%%%%%%%%%%%%%%%%%%%%%%%%%%%%%%%%%%%%%%%%%%%%%%%%%%%%%%%%%	
	\begin{remark}
		\label{Rem:1.2}
		{\rm 
			Let $p=p_F$. It follows from Theorem~\ref{Thm:1.2} that if $\mu$ satisfies
			\[
			\mu({\bf R}^N)\le \gamma_2,
			\]
			then problem \eqref{TFF} possesses a global-in-time solution. As mentioned above, Zhang and Sun \cite{ZhanSun15} showed the global-in-time solvability of problem \eqref{TFF} with $p=p_F$ for initial data whose $L^{1}(\textbf{R}^N)$ norm (scale critical norm) is sufficiently small. We observe that Theorem~\ref{Thm:1.2} leads to the same conclusion as \cite{ZhanSun15}. For the local-in-time solvability, Ghergu, Miyamoto, and Suzuki \cite{GMS22} proved the existence of nonnegative solutions for arbitrary nonnegative $L^1$ initial value in the case of $p=p_F$. 
			Every  measurable function $f \in L^1({\bf R}^N)$ satisfies
			\[
			\lim_{R\to+0} \int_{B(z;R)} |f(y)| \, dy =0
			\]
			for all $z\in{\bf R}^N$. This together with Theorem~\ref{Thm:1.2} implies that for every nonnegative $f\in L^1({\bf R}^N)$, there exists  $T>0$ such that  problem \eqref{TFF} with $u(0)=f$  possesses a solution in $[0,T)\times{\bf R}^N$ in the case of $p=p_F$.
			Therefore, Theorem~\ref{Thm:1.2} leads to the same conclusion as \cite{GMS22}. See also \cite{OkaZhan23}.
		}
	\end{remark}
%%%%%%%%%%%%%%%%%%%%%%%%%%%%%%%%%%%%%%%%%%%%%%%%%%%%%%%%%%%%%%%%%%%%%%%%%%%%%%%%%
%%%%%%%   Remar 1.4     %%%%%%%%%%%%%%%%%%%%%%%%%%%%%%%%%%%%%%%%%%%%%%%%%%%%%%%%%
%%%%%%%%%%%%%%%%%%%%%%%%%%%%%%%%%%%%%%%%%%%%%%%%%%%%%%%%%%%%%%%%%%%%%%%%%%%%%%%%%%	
	\begin{remark}
		%Optimality
		{\rm 
    The power exponents of $\sigma$ appearing in the conditions \eqref{Thm:1.1.1} and \eqref{Thm:1.3.1} given in Theorems~\ref{Thm:1.1} and \ref{Thm:1.3} are optimal.
	 Indeed, let $p>p_F$, $\kappa>0$, and $\mu(x)=\kappa |x|^{-2/(p-1)}$. Then it is easy to see that \eqref{Thm:1.1.1} is violated for a 
        sufficiently large $ \kappa>0$, whence follows the nonexistence of local-in-time solutions of problem \eqref{TFF}. 
      On the other hand, for a sufficiently small $\kappa>0$, \eqref{Thm:1.3.1} is satisfied 
        for any $ r >1$ with $2r/(p-1)<N$ and $ \sigma \in (0,\infty)$, whence follows the existence of a 
          global-in-time solution of problem \eqref{TFF}.
		}
     \end{remark}

		In order to prove Theorems~\ref{Thm:1.2} and \ref{Thm:1.3}, 
		we shall work in the uniformly local Lebesgue space and the local Morrey space, respectively. 
		Recently, Oka and Zhanpeisov \cite{OkaZhan23} constructed a theory of the solvability of problem \eqref{TFF} in the Besov-Morrey space which is a more general concept than the uniformly local Lebesgue and the local Morrey space. Of course, focusing on the existence theory in more general functional spaces, their results are wider than Theorems~\ref{Thm:1.2} and \ref{Thm:1.3}. However, our main concern is the collapse of the solvability which shall be investigated through the life span estimates, rather than the construction of general functional analytic arguments. Therefore, we deduce sufficient conditions so that the existence time of a solution is easily recognizable, whereas in the theory of Oka and Zhanpeisov \cite{OkaZhan23}, the relationship between solvability and the life span of solutions was obscure.

%%%%%%%%%%%%%%%%%%%%%%%%%%%%%%%%%%%%%%%%%%%%%%%%%%%%%%%%%%%%%%%%%%%%%%%%%%%%%%%%%	
	
	%%%%%%%%%%%%%%%%%%%%%%%%%%%%%%%%%%%%%
	\subsection{Main results II (collapse of the solvability)} 
	%%%%%%%%%%%%%%%%%%%%%%%%%%%%%%%%%%%%%

	   As already mentioned above, for the critical case $p=p_F$, the collapse of the 
	    global-in-time solvability  and the local-in-time solvability of problem \eqref{TFF} occurs when $\alpha$ approaches $1-0$.
	    
	    Here, we first discuss the global-in-time solvability of problem \eqref{TFF} with $p=p_F$ as $\alpha\to 1-0$.
	  Let $\mathcal{M}$ be a set of Radon measures on ${\bf R}^N$ such that their norm satisfy $\|\nu|\mathcal{M}\|:=|\nu|({\bf R}^N) <\infty$, where $|\nu|$ is the total variation of $\nu$. 
	  In what follows, we denote $\nu \ge 0$ when $\nu$ is a nonnegative Radon measure on ${\bf R}^N$.
	  For $r>0$, set
	\[
	\mathcal{B}^+(r):=\left\{ ~\! \nu\in\mathcal{M} ~\! ; ~\! \nu \geq 0, 
                                            \ \|\nu|\mathcal{M}\|\le r  ~\! \right\}.
	\]
	Define
	\[
	\mathcal{G}_\alpha:= \{ ~\!  \nu\in \mathcal{M} ~\! ; ~\!  \mbox{\eqref{TFF} with} ~\! u(0)=\nu \geq 0 
                                  \,\,\mbox{possesses a global-in-time solution.} ~\! \}.
	\]
	
	We recall that $\mathcal{G}_\alpha\neq \{0\}$ by Remark \ref{Rem:1.2}. The following result shows that the collapse of the global-in-time solvability of 
	   problem \eqref{TFF} occurs in such a way that $\mathcal{G}_{\alpha}$ shrinks to $\{0\}$ as $\alpha \to 1-0$. Moreover, the explicit estimate of the diameter of $\mathcal{G}_{\alpha}$ is given below. 
%%%%%%%%%%%%%%%%%%%%%%%%%%%%%%%%%%%%%%%%%%%%%%%%%%%%%%%%%%%%%%%%%%%%%%%%%%%%%%%%%%%%%%%%%%%%%%%%
%%%%%%%   Theorem 1.4      %%%%%%%%%%%%%%%%%%%%%%%%%%%%%%%%%%%%%%%%%%%%%%%%%%%%%%%%%%%%%%%%%%%%%
%%%%%%%%%%%%%%%%%%%%%%%%%%%%%%%%%%%%%%%%%%%%%%%%%%%%%%%%%%%%%%%%%%%%%%%%%%%%%%%%%%%%%%%%%%%%%%%%
	\begin{theorem}
		\label{Thm:4.1}
		Let $p=p_F$. Then there exist positive constants $C_1 \le C_2$ independent of $\alpha$ such that, near $\alpha =1$,
		\begin{equation}
		\label{Thm:4.1.1}
		\mathcal{B}^+\left( C_1(1-\alpha)^{\frac{N}{2}} \right)\subset \mathcal{G}_{\alpha}\subset \mathcal{B}^+\left( C_2(1-\alpha)^{\frac{N}{2}} \right).
		\end{equation}
	\end{theorem}
%%%%%%%%%%%%%%%%%%%%%%%%%%%%%%%%%%%%%%%%%%%%%%%%%%%%%%%%%%%%%%%%%%%%%%%%%%%%%%%%%%%%%%%%%%%%%%%%
  We next discuss the collapse of the local-in-time solvability of problem \eqref{TFF} for the doubly critical case, $p=p_F$ and $ \mu \in L^{1}({\bf R}^N)$. For specific considerations, we here focus on the following typical initial data.
	\begin{equation}\label{eq. intro Miyamoto init data}
	f_{\epsilon}(x) :=|x|^{-N} \left|\log |x| \right|^{-\frac{N}{2}-1+\epsilon}\chi_{B(0;e^{-{3/2}})} 
	   \in L^{1}({\bf R}^N) 
	    \quad \text{with} \ 0<\epsilon<N/2.
	\end{equation}
 In fact, problem \eqref{Fjt} with $p=p_F$ and initial data $u(0)=f_{\epsilon}$ admits no nonnegative local-in-time solutions. See \cite{Miya21, BarasPierre85, HisaIshige18}. Let $T_{\alpha}\left[\mu \right]$ be the life span of solutions of problem \eqref{TFF} with the initial data $\mu$. 
   Then, $T_{\alpha}\left[ f_{\epsilon}\right] >0$ by Remark \ref{Rem:1.2}. 
   
    The following result shows that the collapse of the local-in-time solvability of 
	   problem \eqref{TFF} occurs in such a way that $T_{\alpha}\left[ f_{\epsilon}\right]$ 
         converges to zero as $\alpha \to 1-0$ 
	     with the explicit convergence rate given below. 

%%%%%%%%%%%%%%%%%%%%%%%%%%%%%%%%%%%%%%%%%%%%%%%%%%%%%%%%%%%%%%%%%%%%%%%%%%%%%%%%%%%%%%%%%%%%%
%%%%%%   Theorem 1.5     %%%%%%%%%%%%%%%%%%%%%%%%%%%%%%%%%%%%%%%%%%%%%%%%%%%%%%%%%%%%%%%%%%%%
%%%%%%%%%%%%%%%%%%%%%%%%%%%%%%%%%%%%%%%%%%%%%%%%%%%%%%%%%%%%%%%%%%%%%%%%%%%%%%%%%%%%%%%%%%%%%
	\begin{theorem}
		\label{Thm:4.2}
		Let $p=p_F$ and $0<\epsilon<N/2$. Then, there exist $C_1\le C_2$ such that, for each $\kappa>0$, there is $\alpha(\kappa,\epsilon)\in(0,1)$ such that for all $\alpha(\kappa,\epsilon)<\alpha<1$,
		\[
		C_1\kappa^{\frac{2}{N-2\epsilon}}(1-\alpha)^{-\frac{N}{N-2\epsilon}}\le -\log \left( T_{\alpha}[\kappa f_{\epsilon}] \right)\le C_2\kappa^{\frac{2}{N-2\epsilon}}(1-\alpha)^{-\frac{N}{N-2\epsilon}}.
		\]
		In particular, $\lim_{\alpha \to 1-0} T_{\alpha}[\kappa f_{\epsilon}] =0$ for all $\kappa>0$.
	\end{theorem}
%%%%%%%%%%%%%%%%%%%%%%%%%%%%%%%%%%%%%%%%%%%%%%%%%%%%%%%%%%%%%%%%%%%%%%%%%%%%%%%%%%%%%%%%%%%%%%	

	The rest of this paper is organized as follows.
	  In Section~\ref{section:Necessary condition}, we give a proof of Theorem~\ref{Thm:1.1},
        which is essentially based on the method of ordinary differential inequalities introduced 
           by \cite{LaiSie21} (see also \cite{FHIL23, HIT18}).
	  Theorems~\ref{Thm:1.2} and \ref{Thm:1.3} are proved in Section~\ref{section:Sufficient condition}. 
         When $\alpha=1$, Theorems~\ref{Thm:1.2} and \ref{Thm:1.3} coincide with 
           \cite[Theorems 1.3 and 1.4]{HisaIshige18} whose proofs rely on the supersolution method. 
      However, our argument is totally different from theirs. 
       Instead, we introduce the uniformly local $L^q$ space and the local Morrey space $M^{q,\lambda}$ 
         and apply the contraction mapping theorem in these spaces. 
	 In Section~\ref{section. Collapse of the solvability}, we give proofs of 
       Theorems~\ref{Thm:4.1} and \ref{Thm:4.2}.
	%%%%%%%%%%%%%%%%%%%%%%%%%%%%%%%%%%%%%
	%%%%%%%%%%%%%%%%%%%%%%%%%%%%%%%%%%%%%
	\section{Necessary conditions for the solvability.}\label{section:Necessary condition}
	%%%%%%%%%%%%%%%%%%%%%%%%%%%%%%%%%%%%%
	%%%%%%%%%%%%%%%%%%%%%%%%%%%%%%%%%%%%%
	In this section, we give a proof of Theorem~\ref{Thm:1.1}. We follow the argument \cite{FHIL23,HIT18,LaiSie21} in which an inequality related to
	\[
	\int_{{\bf R}^N} G(t^\alpha,x) u(t,x) \,dx.
	\]
	is essential. For the application of this technique, we should estimate the effect of the function $h_{\alpha}$, due to the nonlocality of $P_{\alpha}(t)$ and $S_{\alpha}(t)$.
	
	\begin{proof}[Proof of Theorem \ref{Thm:1.1}]
	Let $p>1$. Suppose that $0<\rho< (16^{-1}T)^{\alpha/2}$. 
		Let $\overline{T}\in (T/3, T/2)$ and $\overline{z} \in {\bf R}^N$ be such that
		\begin{equation}
		\label{eq:3.1}
		\begin{split}
		\infty &> u(2\overline{T}, \overline{z} ) \\
		&\ge \alpha \int_0^{2\overline{T}} (2\overline{T}-s)^{\alpha-1}\int_0^\infty\theta h_\alpha(\theta)\\
		&\qquad\qquad\qquad\times\int_{{\bf R}^N} G((2\overline{T}-s)^\alpha \theta, \overline{z} -y) u(s,y)^p \, dyd\theta ds,
		\end{split}
		\end{equation}
		where we used the definition of solutions and the nonnegativity of $P_{\alpha}(t)\mu$. Set
		\begin{equation*}
		U(t):= \int_{{\bf R}^N} G(t^\alpha,x)u(t,x+\overline{z}) \,dx.
		\end{equation*}
		Note that $U(t)\ge 0$ since $u$ is nonnegative.
		
		\noindent\textbf{Step 1.} First, we prove $U(t) <\infty $ for almost all $t\in (\rho^{2/\alpha},T/3)$.
		Application of the Jensen inequality to \eqref{eq:3.1} yields
		\begin{equation}
		\label{eq:3.2}
		\begin{split}
		\infty &> u( 2\overline{T}, \overline{z}) \\
		& \ge C \alpha \overline{T}^{\alpha-1} \int_{\rho^{2/\alpha}}^{2\overline{T}} \int_{1}^{\infty}  h_\alpha(\theta)\int_{{\bf R}^N} G((2\overline{T}-s)^\alpha \theta,y) u(s,y+\overline{z})^p \, dyd\theta ds\\
		& \ge C \alpha \overline{T}^{\alpha-1} \int_{\rho^{2/\alpha}}^{\overline{T}} \int_{1}^{\infty} h_\alpha(\theta)\left(\int_{{\bf R}^N} G((2\overline{T}-s)^\alpha \theta,y) u(s,y+\overline{z}) \, dy\right)^p d\theta ds,
		\end{split}
		\end{equation}
		where we used the definition of solutions and the nonnegativity of $P_{\alpha}(t)\mu$. 
		Since $s\le 2\overline{T}-s\le T$ for $0<s<\overline{T}(<T/2)$, we have
		\begin{equation}
		\label{eq:3.3}
		\begin{split}
		G((2\overline{T}-s)^\alpha \theta,y) 
		&= \frac{1}{(4\pi (2\overline{T}-s)^\alpha \theta)^{\frac{N}{2}} }\exp\left(-\frac{|y|^2}{4(2\overline{T}-s)^\alpha \theta}\right)\\
		&\ge \frac{1}{\theta^\frac{N}{2}}\left( \frac{s}{2\overline{T}-s}\right)^\frac{N\alpha}{2}\frac{1}{(4\pi s^\alpha )^{\frac{N}{2}} }\exp\left(-\frac{|y|^2}{4s^\alpha}\right)\\
		&\ge \frac{C\rho^N}{\theta^\frac{N}{2}T^\frac{N\alpha}{2}}G(s^\alpha,y)\\
		\end{split}
		\end{equation}
		for all $s\in(\rho^{2/\alpha}, \overline{T})$ and $y\in{\bf R}^N$.
		Combining \eqref{eq:3.2} and \eqref{eq:3.3}, we see that there exists a constant 
		$C_*=C_*(N,\alpha, p, \rho, T)>0$ such that
		\begin{equation*}
		\begin{split}
		\infty >  u(2\overline{T},\overline{z} )  &\ge C_* \int_{1}^{\infty} \theta^{-\frac{Np}{2}}h_{\alpha}(\theta)d\theta \int_{\rho^{2/\alpha}}^{\overline{T}} U(s)^p \,ds\ge C_* \int_{\rho^{2/\alpha}}^{T/3} U(s)^p \,ds.
		\end{split}
		\end{equation*}
		This implies that $U(t)<\infty$ for almost all $t\in (\rho^{2/\alpha},T/3)$. Note that
		\[
		1= \int_{0}^{\infty} h_{\alpha}(\theta)d\theta >\int_{1}^{\infty} \theta^{-\frac{Np}{2}} h_{\alpha}(\theta)d\theta>0,
		\]
		since $h_{\alpha}(\theta)>0$.		
		
		\noindent\textbf{Step 2.} For $t\in(\rho^{2/\alpha},T/3)$ and $\overline{z}\in{\bf R}^N$, set
		\begin{equation*}
		V(t):= t^\frac{N\alpha}{2}U(t), \qquad M:= \mu(B(\overline{z};\rho)).
		\end{equation*}
		Here, we claim that
		\begin{equation}
		\label{eq:3.4}
		\begin{split}
		\infty>V(t) 
		&\ge C_1r_1(\alpha) M+	C_2 r_2(\alpha)T^{\alpha-1}\int_{\rho^{2/\alpha}}^t s^{-\frac{N\alpha}{2}(p-1)} V(s)^p \,ds
		\end{split}
		\end{equation}
		for  almost all $t \in(\rho^{2/\alpha},T/3)$, where
		\[
		r_1(\alpha) = E_{\alpha,1} \left(-\frac{1}{2}\right)>0, \quad r_2(\alpha) =\alpha E_{\alpha,\alpha}  \left(-\frac{1}{2}\right)>0, 
		\]
		and $C_1, C_2$ are positive constants depending only on $N$ and $p$.
		By the definition of $E_{\alpha,\beta}(z)$, we see easily that
		\[
		\limsup_{\alpha\to1-0} r_i(\alpha) \in (0,\infty)
		\]
		for $i=1,2$. By Definition \ref{Def:1.1}, we have
		\begin{equation*}
		\begin{split}
		u(t,x+\overline{z}) 
		&= \int_0^\infty \int_{{\bf R}^N}h_\alpha(\theta) G(t^\alpha\theta, x+\overline{z}-y)\,d\mu(y)d\theta\\
		& +\alpha \int_0^t(t-s)^{\alpha-1} [S_\alpha (t-s) u(s)^p](x+\overline{z}) \, ds,
		\end{split}
		\end{equation*}
		for almost all $t\in(\rho^{2/\alpha},T/3)$ and $x\in{\bf R}^N$.
		Multiplying $G(t^\alpha, x)$ by both side and integrating with respect to $x$, we obtain
		\begin{equation}
		\label{eq:2.5}
		\begin{split}
		\infty>U(t) &
		= \int_{{\bf R}^N} \left[\int_0^\infty h_\alpha (\theta) \left(\int_{{\bf R}^N} G(t^\alpha \theta, x+\overline{z}-y)G(t^\alpha,x)\,dx\right)\,d\theta\right]\,d\mu(y)\\
		&+\alpha \int_{{\bf R}^N} G(t^\alpha, x) \left(\int_0^t(t-s)^{\alpha-1}[S_\alpha(t-s)u(s)^p](x+z)\,ds\right)\,dx\\
		&=: I_1 + \alpha I_2,
		\end{split}
		\end{equation}
		for almost all $t\in(\rho^{2/\alpha},T/3)$.
		For $I_1$, by the semigroup property of the heat kernel $G$ and $\mu \ge 0$, we deduce
		\begin{equation}
		\label{eq:2.6}
		\begin{split}
		I_1&= \int_{{\bf R}^N}\int_0^\infty h_\alpha (\theta) G(t^\alpha(1+\theta),\overline{z}-y)\,d\theta d\mu(y)\\
		&\ge C \left(\int_0^\infty h_\alpha (\theta) (1+\theta)^{-\frac{N}{2}}\, d\theta\right) t^{-\frac{N\alpha}{2}} \int_{B(\overline{z};\rho)} \exp\left(-\frac{|\overline{z}-y|^2}{4t^\alpha}\right) d\mu(y)\\
		&\ge C \left(\int_0^\infty h_\alpha (\theta) (1+\theta)^{-\frac{N}{2}}\,d\theta\right) t^{-\frac{N\alpha}{2}} M,
		\end{split}
		\end{equation}
		for almost all $t\in(\rho^{2/\alpha},T/3)$.
		For $I_2$, we similarly use the semigroup property to get
		\begin{equation}
		\label{eq:2.7}
		\begin{split}
		I_2 = \int_0^t (t-s)^{\alpha-1} \int_{{\bf R}^N}\int_0^\infty \theta h_\alpha(\theta) G(t^\alpha+(t-s)^\alpha\theta, \overline{z}-y)u(s,y)^p\, d\theta dyds.
		\end{split}
		\end{equation}
		Since $t>s$, we have
		\begin{equation*}
		\begin{split}
		&\exp\left(-\frac{|y-\overline{z}|^2}{4t^\alpha + 4(t-s)^\alpha \theta}\right) \ge \exp\left(-\frac{|y-\overline{z}|^2}{4s^\alpha}\right),\\
		&t^\alpha + (t-s)^\alpha \theta \le t^\alpha(1+\theta) = s^\alpha \left(\frac{t}{s}\right)^\alpha (1+\theta).
		\end{split}
		\end{equation*}
		This together with \eqref{eq:2.5}, \eqref{eq:2.6}, \eqref{eq:2.7}, and the Jensen inequality implies that
		\begin{equation}
		\label{eq:2.8}
		\begin{split}
		\infty>U(t) &\ge C_1 \left(\int_0^\infty h_\alpha (\theta) (1+\theta)^{-\frac{N}{2}} \, d\theta \right) t^{-\frac{N\alpha}{2}}M\\
		& + C_2\alpha  T^{\alpha-1} \left(\int_{0}^\infty \theta h_\alpha(\theta) (1+\theta)^{-\frac{N}{2}}\,d\theta\right) t^{-\frac{N\alpha}{2}} \int_{\rho^{2/\alpha}}^t s^{\frac{N\alpha}{2}} U(s)^p\,ds,
		\end{split}
		\end{equation}
		for almost all $t\in(\rho^{2/\alpha},T/3)$.
		Since the function $\theta \mapsto e^{-\theta/2}(1+\theta)^{N/2}$ is bounded on $[0,\infty)$, we use \eqref{eq:1.4.5} to deduce 
		\begin{equation*}
		\begin{split}
		&\int_0^\infty h_\alpha (\theta) (1+\theta)^{-\frac{N}{2}} \, d\theta \ge C\int_0^\infty h_\alpha (\theta) e^{-\frac{\theta}{2}} \, d\theta = CE_{\alpha,1}\left(-\frac{1}{2}\right) = Cr_1(\alpha),\\
		&\alpha \int_{0}^\infty \theta h_\alpha(\theta) (1+\theta)^{-\frac{N}{2}}\,d\theta \ge C\alpha \int_0^\infty \theta h_\alpha (\theta) e^{-\frac{\theta}{2}} \, d\theta = C\alpha E_{\alpha,\alpha}\left(-\frac{1}{2}\right) =C r_2(\alpha),\\
		\end{split}
		\end{equation*}
		for some $C=C(N)>0$. Therefore \eqref{eq:2.8} yields \eqref{eq:3.4}.
		
		\noindent\textbf{Step 3.} We may then let $\zeta$ denote the unique local solution of the integral equation
		\begin{equation}
		\label{eq:3.5}
		\zeta(t) = C_1r_1(\alpha) M+	C_2 r_2(\alpha)T^{\alpha-1}\int_{\rho^{2/\alpha}}^t s^{-\frac{N\alpha}{2}(p-1)} \zeta(s)^p \,ds
		\end{equation}
		for  $t \in[\rho^{2/\alpha},T/4]$. 
		Hence, $\zeta$ is the unique local solution of 
		\begin{equation*}
		\zeta'(t) = C_2 r_2(\alpha) T^{\alpha-1}t^{-\frac{N\alpha}{2}(p-1)} \zeta(t)^p, \qquad \zeta(\rho^\frac{2}{\alpha}) = C_1r_1(\alpha) M.
		\end{equation*}
		
		Applying the standard theory for ordinary differential equations to \eqref{eq:3.5},
		we see from  \eqref{eq:3.4}  that the solution $\zeta$ exists in $[\rho^{2/\alpha}, T/4]$.
		It is easy to solve this ordinary differential equation and get
		\[
		C_2r_2(\alpha) T^{\alpha-1} \int_{\rho^{2/\alpha}}^{T/4} t^{-\frac{N\alpha}{2}(p-1)} \, dt = \int_{\zeta(\rho^{2/\alpha})}^{\zeta(T/4)} \zeta^{-p}\,d\zeta 
		\le \int_{\zeta(\rho^{2/\alpha})}^{\infty} \zeta^{-p}\,d\zeta. 
		\]
		This implies that
		\[
		\mu(B(\overline{z};\rho))=M\le \frac{C}{r_1(\alpha) r_2(\alpha)^\frac{1}{p-1}} T^{-\frac{\alpha-1}{p-1}} \left(\int_{\rho^{2/\alpha}}^{T/4} t^{-\frac{N\alpha}{2}(p-1)} \, dt \right)^{-\frac{1}{p-1}}
		\]
		for all  $\rho\in (0, (T/16)^{\alpha/2})$. 
		Setting $\sigma:= 4^\alpha \rho$, we obtain
		\begin{equation}
		\label{eq:3.6}
		\mu(B(\overline{z};4^{-\alpha}\sigma)) \le \frac{C}{r_1(\alpha) r_2(\alpha)^\frac{1}{p-1}} T^{-\frac{\alpha-1}{p-1}} \left(\int_{\sigma^{2/\alpha}/16}^{T/4} t^{-\frac{N\alpha}{2}(p-1)} \, dt \right)^{-\frac{1}{p-1}}
		\end{equation}
		for all $\sigma \in (0,T^{\alpha/2})$. 
		Let $z\in{\bf R}^N$.
		Then we can take integer $m=m(N,\alpha)$ satisfying $\limsup_{\alpha\to1-0} m(\alpha)<\infty$
		and $\{\overline{z}_i\}_{i=1}^m\subset {\bf R}^N$  satisfying \eqref{eq:3.6} such that 
		\begin{equation*}
		\begin{split}
		\mu(B(z;\sigma)) 
		&\le \sum_{i=1}^m \mu(B(\overline{z}_i;4^{-\alpha}\sigma)) \\
		&\le \frac{Cm}{r_1(\alpha) r_2(\alpha)^\frac{1}{p-1}} T^{-\frac{\alpha-1}{p-1}} \left(\int_{\sigma^{2/\alpha}/16}^{T/4} t^{-\frac{N\alpha}{2}(p-1)} \, dt \right)^{-\frac{1}{p-1}}
		\end{split}
		\end{equation*}
		for all $\sigma \in (0,T^{\alpha/2})$. 
		Since $z\in{\bf R}^N$ is arbitrary, we obtain \eqref{Thm:1.1.2}.
		Furthermore, applying the above argument with $T = \sigma^{2/\alpha}$, we obtain \eqref{Thm:1.1.1}.
	\end{proof}

	%%%%%%%%%%%%%%%%%%%%%%%%%%%%%%%%%%%%%
	%%%%%%%%%%%%%%%%%%%%%%%%%%%%%%%%%%%%%
	\section{Sufficient conditions for the solvability.}\label{section:Sufficient condition}
	%%%%%%%%%%%%%%%%%%%%%%%%%%%%%%%%%%%%%
	%%%%%%%%%%%%%%%%%%%%%%%%%%%%%%%%%%%%%

	%%%%%%%%%%%%%%%%%%%%%%%%%%%%%%%%%%%%%
	\subsection{Preliminaries.}
	%%%%%%%%%%%%%%%%%%%%%%%%%%%%%%%%%%%%%
	To prove Theorem~\ref{Thm:1.2},
	we introduce the uniformly local  $L^q$ space.
	For details, see e.g. \cite{MaeTera06, FujiIoku18, QuitSoup19, ARCD2004}.
	\begin{definition}
	Let $L^p_{{\rm uloc}}=L^p_{{\rm uloc}}({\bf R}^N)$ {\rm ( $1\le p\le\infty$ )} 
          be the set of all measurable functions $f$ 
            defined in ${\bf R}^N$ such that
		\[
		\|f|L^p_{{\rm uloc}}\| := \sup_{z\in{\bf R}^N} \|f|L^p(B(z;1))\|<\infty.
		\]
	\end{definition}
	 We note that $L^p_{{\rm uloc}}({\bf R}^N)$ with norm $\|\cdot|L^p_{{\rm uloc}}\|$ forms a Banach space. 
	   Furthermore, $M^1=M^1({\bf R}^N)$ is defined as the set of all Radon measures
          $\mu$ on ${\bf R}^N$ such that
	\[
	\|\mu|M^1\| := \sup_{z\in{\bf R}^N} |\mu|(B(z;1)) <\infty,
	\]
	where $|\mu|$ is the total variation of $\mu$. 
	To prove Theorem~\ref{Thm:1.2}, we heavily rely on the following proposition and \eqref{MPMQ lambdaN} below. 
%%%%%%%%%%%%%%%%%%%%%%%%%%%%%%%%%%%%%%%%%%%%%%%%%%%%%%%%%%%%%%%%%%%%%%%%%%%%%%%%%%
%%%%%%%%   Proposition 3.1     %%%%%%%%%%%%%%%%%%%%%%%%%%%%%%%%%%%%%%%%%%%%%%%%%%%
%%%%%%%%%%%%%%%%%%%%%%%%%%%%%%%%%%%%%%%%%%%%%%%%%%%%%%%%%%%%%%%%%%%%%%%%%%%%%%%%%%
	\begin{proposition}
		\label{LpLq} 
		Let $1\le p\le \infty$. One has
		\begin{equation}
		\|e^{t\Delta} \mu|L^p_{{\rm uloc}}\| \le C\left(1+t^{-\frac{N}{2}(1- \frac{1}{p})}\right) \|\mu|M^1\|
		\end{equation}
		for $\mu\in M^1({\bf R}^N)$ and $t>0$. 
	\end{proposition}
%%%%%%%%%%%%%%%%%%%%%%%%%%%%%%%%%%%%%%%%%%%%%%%%%%%%%%%%%%%%%%%%%%%%%%%%%%%%%%%%
	\begin{proof}
		We can show the estimate by the modified proof of \cite[Theorem 3.1]{MaeTera06}. We just state the essential part. The key estimates are equations (3.13) and (3.14) in \cite[page 382]{MaeTera06} which are summarized as follows:
		\[
		\begin{aligned}
		&\left\| e^{t\Delta}\mu | L^p\left( S\left( k,1/2 \right) \right) \right\|\\
		&\le\left\| \sum_{k',k''\in \textbf{Z}^N} \left( \chi_{S\left( k',1/2 \right)}G(t,\cdot)\right) \ast \left( \chi_{S\left( k'',1/2 \right)}\mu \right)  |L^p\left( S\left( k,1/2 \right) \right)\right\|\\
		&\le \sum_{k'\in \textbf{Z}^N} 3^N \left\| \chi_{S\left( k',1/2 \right)}G(t,\cdot) |L^p\right\| \sup_{k''\in \textbf{Z}^N} |\mu|\left( S\left( k'',1/2 \right) \right),
		\end{aligned}
		\]
		where
		\[
		S(\zeta, l):= \left\{ y; \max_{1\le i\le N} |y_i-\zeta_i|\le l \right\}.
		\]
		For $p=1$ and $p=\infty$, we can justify the above estimate for $\mu \in M^1(\textbf{R}^N)$. For general $p\in (1,\infty)$, we use the interpolation argument.
	\end{proof}
%%%%%%%%%%%%%%%%%%%%%%%%%%%%%%%%%%%%%%%%%%%%%%%%%%%%%%%%%%%%%%%%%%%%%%%%%%%%%%%
	
	In order to prove Theorem~\ref{Thm:1.3}, we introduce the local Morrey space. For details, see \cite{GM89, Kato92, KozoYama94}.
%%%%%%%%%%%%%%%%%%%%%%%%%%%%%%%%%%%%%%%%%%%%%%%%%%%%%%%%%%%%%%%%%%%%%%%%%%%%%%%
%%%%%%    Deinition 3.2       %%%%%%%%%%%%%%%%%%%%%%%%%%%%%%%%%%%%%%%%%%%%%%%%%
%%%%%%%%%%%%%%%%%%%%%%%%%%%%%%%%%%%%%%%%%%%%%%%%%%%%%%%%%%%%%%%%%%%%%%%%%%%%%%%
	\begin{definition}
	 Let $1\le q< \infty$ and $0<\lambda \le N$. 
		The local Morrey space $M^{q,\lambda}=M^{q,\lambda}({\bf R}^N)$ is defined as the set of all measurable functions $f$ defined in ${\bf R}^N$ such that
	\begin{equation}\label{def:Morrey}
	  \|f|M^{q,\lambda}\| 
         := \sup_{z\in{\bf R}^N} \sup_{R\in(0,1]} R^{\frac{1}{q}(\lambda-N)} \|f|L^q(B(z;R))\|
            < \infty.
	\end{equation}
	\end{definition}
%%%%%%%%%%%%%%%%%%%%%%%%%%%%%%%%%%%%%%%%%%%%%%%%%%%%%%%%%%%%%%%%%%%%%%%%%%%%%%
	We note that $M^{q,\lambda}({\bf R}^N)$ with norm $\|\cdot|M^{q,\lambda}\|$ forms a Banach space. 

	To prove Theorem~\ref{Thm:1.3}, the following proposition plays an important role.

%%%%%%%%%%%%%%%%%%%%%%%%%%%%%%%%%%%%%%%%%%%%%%%%%%%%%%%%%%%%%%%%%%%%%%%%%%%%
%%%%%   Proposition 3.2.    %%%%%%%%%%%%%%%%%%%%%%%%%%%%%%%%%%%%%%%%%%%%%%%%
%%%%%%%%%%%%%%%%%%%%%%%%%%%%%%%%%%%%%%%%%%%%%%%%%%%%%%%%%%%%%%%%%%%%%%%%%%%%
	\begin{proposition}
		\label{MpMq} 
		Let $1\le q \le p <\infty$ and $0< \lambda \le N$. Then there exists a constant $C>0$ such that
		\begin{equation}
		\label{MPMQ1}
		\|e^{t\Delta} \mu|M^{p,\lambda}\| \le C (1+t^{-\frac{\lambda}{2}(\frac{1}{q}-\frac{1}{p})}) \|\mu|M^{q,\lambda}\|
		\end{equation}
		for $\mu\in M^{q,\lambda}({\bf R}^N)$ and $t>0$. 

	\end{proposition}
%%%%%%%%%%%%%%%%%%%%%%%%%%%%%%%%%%%%%%%%%%%%%%%%%%%%%%%%%%%%%%%%%%%%%%%%%%%
	
	Note that \eqref{MPMQ1} with $\lambda =N$ implies
	\begin{equation}
		\label{MPMQ lambdaN}
		\| e^{t\Delta}\mu | L^{p}_{{\rm uloc}} \|\le C \left( 1+ t^{-\frac{N}{2}\left( \frac{1}{q}-\frac{1}{p} \right)} \right) \| \mu | L^{q}_{{\rm uloc}}\|,
	\end{equation}
	which has been proved in \cite{MaeTera06}.
	In order to prove Proposition \ref{MpMq}, we recall the inhomogeneous Besov--Morrey space introduced by Kozono--Yamazaki \cite{KozoYama94}.
	Let $\zeta(t)$ be a $C^\infty$-function on $[0,\infty)$ such that 
      $0\le \zeta(t)\le 1$, $\zeta(t)\equiv1$ for $t \in [0, 3/2]$, 
         and $\mbox{supp}\,\zeta \subset[0,5/3)$.
	For every $j\in{\bf Z}$, put $\varphi(\xi) := \zeta(2^{-j}|\xi|)-\zeta(2^{1-j}|\xi|)$ and $\varphi_{(0)}(\xi):=\zeta(|\xi|)$.
	Then we have $\varphi_{j}, \varphi_{(0)}\in C_{0}^\infty({\bf R}^N)$ for all $j\in{\bf Z}$, and 
	\[
	\varphi_{(0)}(\xi) + \sum_{j=1}^\infty \varphi_j(\xi) =1 \quad \mbox{for all} \quad \xi\in{\bf R}^N.
	\]

%%%%%%%%%%%%%%%%%%%%%%%%%%%%%%%%%%%%%%%%%%%%%%%%%%%%%%%%%%%%%%%%%%%%%%%%%%%
%%%%%%%  Definition  3.3.     %%%%%%%%%%%%%%%%%%%%%%%%%%%%%%%%%%%%%%%%%%%%%
%%%%%%%%%%%%%%%%%%%%%%%%%%%%%%%%%%%%%%%%%%%%%%%%%%%%%%%%%%%%%%%%%%%%%%%%%%%	
	\begin{definition}
		Let $1\le q \le p <\infty$, $1\le r \le \infty$, and $s\in{\bf R}$.
	 	  The local Besov--Morrey space $N^s_{p,q,r}=N^s_{p,q,r}({\bf R}^N)$ is defined as 
            the set of all tempered distributions  $f\in \mathcal{S}'$ such that
		\[
		\|f|{N^s_{p,q,r}}\| : = \|\mathcal{F}^{-1}\varphi_{(0)}(\xi)\mathcal{F}f|M^{q,\lambda}\| + \left\| \{2^{sj} \|\mathcal{F}^{-1} \varphi_j(\xi) \mathcal{F}f|M^{q,\lambda}\|\}_{j=1}^\infty|l^r \right\| <\infty,
		\]
		where $\lambda := qN/p \le N$ and $\mathcal{F}$ denotes the Fourier transform on ${\bf  R}^N$.
		Here, $l^r$ denotes the usual sequence space.
		
	\end{definition}
%%%%%%%%%%%%%%%%%%%%%%%%%%%%%%%%%%%%%%%%%%%%%%%%%%%%%%%%%%%%%%%%%%%%%%%%%%%
	
	Next, we prepare two lemmas to prove Proposition \ref{MpMq}, which have been proved by Kozono--Yamazaki \cite[Theorem 2.5, Proposition 2.11, and Theorem 3.1]{KozoYama94}.
	
	\begin{lemma}
		Let $1\le q'\le p' <\infty$, $1\le r\le\infty$,  $s\in{\bf R}$, and $\theta\in(0,1)$.
		Then the following embeddings are continuous:
		\begin{equation}
		\label{em1}
		N^{s}_{p',q',r} \hookrightarrow N^{s-\frac{N}{p'}(1-\theta)}_{\frac{p'}{\theta}, \frac{q'}{\theta}, r},
		\end{equation}
		\begin{equation}
		\label{em2}
		N^{0}_{p',q',1} \hookrightarrow M^{q', \frac{q'N}{p'}} \hookrightarrow N^{0}_{p',q',\infty}.
		\end{equation}
	\end{lemma}
	
	\begin{lemma}
		Let $s<\sigma$ and $1\le q' \le p' <\infty$. Then there exists a constant $C>0$ such that
		\begin{equation}
		\label{smooth}
		\|e^{t\Delta} f|{N^\sigma_{p',q',1}}\| \le C (1+t^{\frac{s-\sigma}{2}}) \|f|{N^s_{p',q',\infty}}\|
		\end{equation}
		for all $f\in N^s_{p',q',\infty} $ and $t>0$.
	\end{lemma}
	
	\begin{proof}[Proof of Proposition \ref{MpMq}.]
		First, we prove \eqref{MPMQ1} in the case of $p=q$.
		Let $\mu \in M^{p,\lambda}$, $z\in {\bf R}^N$, and $R\in(0,1]$. 
		By \eqref{def:Morrey} and the Jensen inequality, we have 
		\begin{equation*}
		\begin{split}
		\int_{B(z;R)} |e^{t\Delta} \mu (x)|^p\, dx
		& = \int_{B(z;R)}\left|\int_{{\bf R}^N} G(x-y,t) \mu(y)\,dy\right|^p\, dx\\
		& \le\int_{B(z;R)} \int_{{\bf R}^N} G(x-y,t) |\mu(y)|^p\,dydx\\
		& =   \int_{B(z;R)}\int_{{\bf R}^N} G(y,t) |\mu(x-y)|^p\,dydx\\
		&= \int_{{\bf R}^N} G(y,t) \int_{B(z;R)}|\mu(x-y)|^p\,dxdy\\
		&\le R^{N-\lambda} \|\mu|M^{p,\lambda}\|^p \int_{{\bf R}^N} G(y,t)\,dy\\
		&=  R^{N-\lambda} \|\mu|M^{p,\lambda}\|^p.
		\end{split}
		\end{equation*}
		Since $z\in{\bf R}^N$ and $R\in(0,1]$ are arbitrary, we have
		\[
		\|e^{t\Delta} \mu (x)|M^{p,\lambda}\| \le  \|\mu|M^{p,\lambda}\|,
		\]
		whence follows \eqref{MPMQ1} with $p=q$.

		Next, we prove \eqref{MPMQ1} with $p>q$.
		By \eqref{em2}  with 
		\[
		p'= \frac{pN}{\lambda} \quad \mbox{and} \quad q' = p,
		\]
		%($A\to B$ means "substituting $B$ into $A$"),
		we have 
		\begin{equation}
		\label{eq:em4}
		\begin{split}
		\|e^{t\Delta} \mu|M^{p,\lambda}\|
		& \le C \left\|e^{t\Delta} \mu |{N^0_{\frac{pN}{\lambda}, p,1}}\right\|\\
		\end{split}
		\end{equation}
		for all $t>0$. Note that
		\[
		N^{0}_{\frac{pN}{\lambda}, p,1} = 
		N^{N(1-\frac{q}{p})\frac{\lambda}{qN}-N(1-\frac{q}{p})\frac{\lambda}{qN}}_{\frac{qN}{\lambda}\frac{p}{q}, q \frac{p}{q},1}.
		\]
		By \eqref{em1} with
		\[
		\theta = \frac{q}{p}\in(0,1), \quad p'=\frac{qN}{\lambda}, \quad q'= q, \quad r=1,\quad s=N\left(1-\frac{q}{p}\right)\frac{\lambda}{qN},
		\]
		we have
		\begin{equation}
		\label{eq:em5}
		\left\|e^{t\Delta} \mu|{N^0_{\frac{pN}{\lambda}, p,1}}\right\| \le C \left\|e^{t\Delta} \mu|{N^{\lambda(\frac{1}{q}-\frac{1}{p})}_{\frac{qN}{\lambda}, q,1}}\right\|
		\end{equation}
		for all $t>0$.
		By  \eqref{em2} and \eqref{smooth} with 
		\[
		p'= \frac{qN}{\lambda}, \quad q'= q, \quad s=0, \quad \sigma = \lambda\left(\frac{1}{q}-\frac{1}{p}\right)>0,
		\]
		\eqref{eq:em4}, and \eqref{eq:em5}, we have 
		\begin{equation*}
		%\label{eq:3.5}
		\begin{split}
		\|e^{t\Delta} \mu|M^{p,\lambda}\| \le C  \left\|e^{t\Delta} \mu|{N^{\lambda(\frac{1}{q}-\frac{1}{p})}_{\frac{qN}{\lambda}, q,1}}\right\|
		&\le C (1+t^{-\frac{\lambda}{2}(\frac{1}{q}-\frac{1}{p})})\left\| \mu|{N^{0}_{\frac{qN}{\lambda}, q,\infty}}\right\|\\
		&\le C (1+t^{-\frac{\lambda}{2}(\frac{1}{q}-\frac{1}{p})})\| \mu|M^{q,\lambda}\|\\
		\end{split}
		\end{equation*}
		for all $t>0$. Thus, we arrive at \eqref{MPMQ1}.
	\end{proof}
	%%%%%%%%%%%%%%%%%%%%%%%%%%%%%%%%%%%%%
	\subsection{Proof of Theorem \ref{Thm:1.2}.}
	%%%%%%%%%%%%%%%%%%%%%%%%%%%%%%%%%%%%%
	
	In this subsection, we give a proof of Theorem~\ref{Thm:1.2}. 
	\begin{proof}[Proof of Theorem \ref{Thm:1.2}.]
	Let $1<p\le p_F$.
	The proof is based on the contraction mapping theorem.
		It suffices to consider the case of $T=1$. Indeed, for any solution $u$ of \eqref{TFF} in $[0,T)\times{\bf R}^N$, where $T\in(0,\infty)$, we see that $u_\lambda (x,t) := \lambda^{2\alpha/(p-1)}u(\lambda^2 t , \lambda^\alpha x)$ with $\lambda:= \sqrt{T}$ is a solution of \eqref{TFF} in $[0,1)\times{\bf R}^N$.
		Then \eqref{Thm:1.2.1} reads
		\begin{equation}\label{eq. cond. of thm1.2 with T=1}
			\sup_{z\in{\bf R}^N}\mu(B(z;1))\le \gamma_2,
		\end{equation}
		where $\gamma_2$ is the constant as in Theorem~\ref{Thm:1.2}.
		Set
		\[X :=L^\infty((0,1);L^{p}_{{\rm uloc}}({\bf R}^N))
		\]
		and
		\[
		\|u|X\|:= \sup_{t\in(0,1)} t^\beta \|u(t)|L^p_{{\rm uloc}}\|, \quad \beta := \frac{N\alpha}{2p}(p-1) \le \frac{\alpha}{p}.
		\]
		For $u\in X$, define
		\[
		\Phi[u](t):= P_\alpha(t)\mu +\alpha\int_0^t (t-s)^{\alpha-1} S_\alpha(t-s)|u(s)|^p \,ds.
		\]
		Denote $B_M := \{u\in X;\|u|X\|\le M\}$ for $M>0$, which will be fixed later, and set 
		\[
		d(u,v) := \|u-v|X\|  \qquad \forall u, v \in B_M.
		\]
		 Then $(B_M,d)$ becomes a complete metric space.
		   We are going to show below that $\Phi$ maps $B_M$ into itself for all $t \in (0,1)$.
		   
           Take any $u\in B_M$ and $t\in (0,1)$, then since
		\[
		-\frac{N}{2}\left(1-\frac{1}{p}\right)>-1
		\]
		for $1<p\le p_F$, \eqref{eq:1.4} yields
		\[
		\int_0^\infty h_\alpha(\theta) (1+\theta^{-\frac{N}{2}(1-\frac{1}{p})} )\, d\theta= 1+\frac{\Gamma\left(1-\frac{N}{2}(1-\frac{1}{p})\right)}{\Gamma\left(1-\frac{N\alpha}{2}(1-\frac{1}{p})\right)} <\infty
		\]
		and
		\[
		\int_0^\infty h_\alpha(\theta) (\theta+\theta^{1-\frac{N}{2}(1-\frac{1}{p})} )\, d\theta=\frac{\Gamma\left(2\right)}{\Gamma\left(1+\alpha\right)}+\frac{\Gamma\left(2-\frac{N}{2}(1-\frac{1}{p})\right)}{\Gamma\left(1+\alpha-\frac{N\alpha}{2}(1-\frac{1}{p})\right)} <\infty.
		\]
		Note that the constants above are uniformly bounded near $\alpha=1$, and that $\|f^p|L^1_{{\rm uloc}}\| =\|f|L^p_{{\rm uloc}}\|^p$ for all $f\in L^p_{{\rm uloc}}({\bf R}^N)$.
		It follows from Proposition \ref{LpLq} and \eqref{eq. cond. of thm1.2 with T=1} that
		\begin{equation*}
		\begin{split}
		t^\beta \|P_\alpha (t) \mu |L^p_{{\rm uloc}}\| 
		&\le t^\beta \int_0^\infty h_\alpha(\theta) \|e^{t^\alpha \theta \Delta}\mu|L^p_{{\rm uloc}}\| \, d\theta\\
		&\le C \|\mu|M^1\|t^\beta\int_0^\infty h_\alpha(\theta )(1+(t^\alpha \theta)^{-\frac{N}{2}(1-\frac{1}{p})})\, d\theta \\
		&\le C\|\mu|M^1\|t^{\beta-\frac{N\alpha}{2}(1-\frac{1}{p})} \int_0^\infty h_\alpha(\theta )(1+\theta^{-\frac{N}{2}(1-\frac{1}{p})})\, d\theta \\
		&\le C_1 \gamma_2
		\end{split}
		\end{equation*}
		for $0<t<1$. Similarly, by \eqref{MPMQ lambdaN},
		\begin{equation*}
		\begin{split}
		&t^\beta \left\| \alpha\int_0^t (t-s)^{\alpha-1} S_\alpha(t)|u(s)|^p\,ds |L^p_{{\rm uloc}}\right\| \\
		&\le \alpha t^\beta \int_0^t (t-s)^{\alpha-1} \|S_\alpha(t)|u(s)|^p|L^p_{{\rm uloc}}\|\,ds \\
		&\le C  t^\beta  \int_0^t (t-s)^{\alpha-1} \int_0^\infty \theta h_\alpha(\theta) (1+[(t-s)^\alpha \theta]^{-\frac{N}{2}(1-\frac{1}{p})})\||u(s)|^p|L^1_{{\rm uloc}}\| \,d\theta ds\\
		&\le CM^pt^{\alpha(1-\frac{N}{2}(p-1))} B\left(\alpha-\frac{N\alpha}{2}\left(1-\frac{1}{p}\right), 1-\frac{N\alpha}{2}(p-1)\right)\\
		&\le C_2M^pB\left(\alpha-\frac{N\alpha}{2}\left(1-\frac{1}{p}\right), 1-\frac{N\alpha}{2}(p-1)\right)
		\end{split}
		\end{equation*}
		for $0<t<1$.  
		Note that $C_1$ and $C_2$ are independent of $\alpha$. For $c_*>0$, set
		\begin{equation}\label{gamma_2}
		\gamma_2= c_* B\left(\alpha-\frac{N\alpha}{2}\left(1-\frac{1}{p}\right), 1-\frac{N\alpha}{2}(p-1)\right)^{-\frac{1}{p-1}}
		\end{equation}
		and
		\[
		M= 2C_1 c_*B\left(\alpha-\frac{N\alpha}{2}\left(1-\frac{1}{p}\right), 1-\frac{N\alpha}{2}(p-1)\right)^{-\frac{1}{p-1}}.
		\]
		Taking  a sufficiently small $c_*>0$ if necessary, we have by the above inequalities
		\begin{equation*}
		%\label{sub:2}
		\begin{split}
		\|\Phi[u]|X\| 
		&= \sup_{t\in(0,1)} t^\beta\|\Phi[u](t)|L^p_{{\rm uloc}}\| \\
		&\le (C_1c_*+ 2^pC_1^pC_2c_*^p)B\left(\alpha-\frac{N\alpha}{2}\left(1-\frac{1}{p}\right), 1-\frac{N\alpha}{2}(p-1)\right)^{-\frac{1}{p-1}} \\
		&\le 2C_1c_*B\left(\alpha-\frac{N\alpha}{2}\left(1-\frac{1}{p}\right), 1-\frac{N\alpha}{2}(p-1)\right)^{-\frac{1}{p-1}} = M.
		\end{split}
		\end{equation*}
		This implies that $\Phi(B_M) \subset B_M$. 
		
		Next, we show that $\Phi: B_M \to B_M$ is a contraction mapping. 
		Take any $u,v\in B_M$, then by manipulations similar to those above, we have
		\begin{equation}
		\label{sub:4}
		\begin{split}
		& t^\beta \|\Phi[u](t)-\Phi[v](t)|L^p_{{\rm uloc}}\|\\
		&\le \alpha t^\beta \int_0^t (t-s)^{\alpha-1} \int_0^\infty \theta h_\alpha(\theta) \|e^{(t-s)^\alpha \theta \Delta}(|u(s)|^p-|v(s)|^p)|L^p_{{\rm uloc}}\| \, d\theta ds\\
		&\le  C t^\beta \int_0^t (t-s)^{\alpha-1-\frac{N\alpha}{2}(1-\frac{1}{p})}  \||u(s)|^p-|v(s)|^p|L^1_{{\rm uloc}}\| \, ds\\
		&\qquad\qquad\qquad\times \int_0^\infty  h_\alpha(\theta)(\theta+ \theta^{1-\frac{N}{2}(1-\frac{1}{p})}) \,d\theta\\
		&\le  C t^\beta \int_0^t (t-s)^{\alpha-1-\frac{N\alpha}{2}(1-\frac{1}{p})}  (\|u(s)|L^p_{{\rm uloc}}\|^{p-1} + \|v(s)|L^p_{{\rm uloc}}\|^{p-1})\\
		&\qquad\qquad\qquad\times \|u(s)-v(s)|L^p_{{\rm uloc}}\| \, ds\\
		&\le  C t^\beta  M^{p-1}\int_0^t (t-s)^{\alpha-1-\frac{N\alpha}{2}(1-\frac{1}{p})} s^{-p\beta}\, ds \, d(u,v)\\
		&\le C_3M^{p-1}t^{\alpha(1-\frac{N}{2}(p-1))}B\left(\alpha-\frac{N\alpha}{2}\left(1-\frac{1}{p}\right), 1-\frac{N\alpha}{2}(p-1)\right)d(u,v)
		\end{split}
		\end{equation}
		for $0<t<1$. Note that $C_3$ is independent of $\alpha$. 
		Taking a sufficiently small $c_*>0$ if necessary, we have by \eqref{sub:4}
		\[
		d(\Phi[u], \Phi[v]) \le 2^{p-1} C_1^{p-1}C_3 c_*^{p-1} d(u,v) \le \frac{1}{2} d(u,v)
		\] 
		for $u,v\in B_M$. Note that $c_*$ is independent of $\alpha$.
		This implies that $\Phi: B_M \to B_M$ is a contraction mapping. 
		Therefore, by the contraction mapping theorem, we know $\Phi$ has a unique fixed point $u\in B_M$. 
		Obviously, if $\mu$ is a nonnegative Radon measure on ${\bf R}^N$, we see that $u(x,t)\ge0$ for almost all $t\in(0,1)$ and $x\in{\bf R}^N$.
		This implies that $u$ is a solution of problem \eqref{TFF} in $[0,1)\times{\bf R}^N$. Furthermore \eqref{est:gamma2:sub} easily follows from \eqref{gamma_2}.
		
		Finally, let $p=p_F$. By the property of the Beta function:
		\[
		B(x,y)=\frac{\Gamma(x)\Gamma(y)}{\Gamma(x+y)} \quad \mbox{for all} \quad x,y>0
		\]
		and the Gamma function:
		\[
		\Gamma(1+x)=x\Gamma(x) \quad \mbox{for all}\quad x>0,
		\]
		we obtain
		\[
		\begin{aligned}
		B\left( \alpha -\frac{N\alpha}{2}\left(1-\frac{1}{p}\right), 1-\alpha \right)
		&=\frac{\Gamma\left( \alpha -\frac{\alpha N}{2}\left( 1-\frac{1}{p} \right) \right)}{\Gamma\left( 1 -\frac{\alpha N}{2}\left( 1-\frac{1}{p} \right) \right)} \Gamma(1-\alpha)\\
		&=\frac{\Gamma\left( \alpha -\frac{\alpha}{p} \right)\Gamma(2-\alpha) }{\Gamma\left(1 -\frac{\alpha}{p}\right)}
		(1-\alpha)^{-1},
		\end{aligned}
		\]
		since $N(p-1)/2=1$ if $p=p_F$. Combining this equality with \eqref{gamma_2} and $p=p_F$ yields 
           \eqref{est:gamma2:critical}. 
	\end{proof}
	
	%%%%%%%%%%%%%%%%%%%%%%%%%%%%%%%%%%%%%
	\subsection{ Proof of Theorem \ref{Thm:1.3}.}
	%%%%%%%%%%%%%%%%%%%%%%%%%%%%%%%%%%%%%
	 Here, we give a proof of Theorem~\ref{Thm:1.3}. 

%%%%%%%%%%%%%%%%%%%%%%%%%%%%%%%%%%%%%%%%%%%%%%%%%%%%%%%%%%%%%%%%%%%%%%%%%%%%%%%%%%%%%%
%%%%%%%%   Proof of Theorem 1.3.     %%%%%%%%%%%%%%%%%%%%%%%%%%%%%%%%%%%%%%%%%%%%%%%%%
%%%%%%%%%%%%%%%%%%%%%%%%%%%%%%%%%%%%%%%%%%%%%%%%%%%%%%%%%%%%%%%%%%%%%%%%%%%%%%%%%%%%%%	
	\begin{proof}[Proof of Theorem \ref{Thm:1.3}.]
	Let $p>p_F$.
Assume \eqref{Thm:1.3.1}. Without loss of generality, we can assume  that
 \begin{equation}
 \label{assumption:r}
 r\in (1,p) \quad \mbox{and} \quad \frac{2r}{p-1} < N.
 \end{equation}
Indeed, if $r>1$ does not satisfy \eqref{assumption:r}, then, for any $r'\in(1,r)$ satisfying \eqref{assumption:r}, we apply the H\"{o}lder inequality to obtain
\begin{equation*}
\begin{split}
\sup_{z\in {\bf R}^N} \int_{B(z;\sigma)} \mu(y)^{r'} \, dy 
&\le \sup_{z\in{\bf R}^N} \left[\int_{B(z;\sigma)} \,dy\right]^{1-\frac{r'}{r}} \left[\int_{B(z;\sigma)} \mu(y)^r\, dy\right]^\frac{r'}{r}\\
& \le C\gamma_3^{r'} \sigma^{N-\frac{2 r'}{p-1}}
\end{split}
\end{equation*}
for all $\sigma \in (0,T^{\alpha/2})$.
Then \eqref{Thm:1.3.1} holds with $r$ replaced by $r'$.
	Similarly to the proof of Theorem~\ref{Thm:1.2}, it suffices to consider the case of $T=1$, because if $u$ is a solution of problem \eqref{TFF} in $[0,T)\times{\bf R}^N$, we see that $u_\lambda (x,t) := \lambda^{2\alpha/(p-1)}u(\lambda^2 t , \lambda^\alpha x)$ with $\lambda:= \sqrt{T}$ is a solution of problem \eqref{TFF} in $[0,1)\times{\bf R}^N$.  Hence, in view of \eqref{def:Morrey}, we see that \eqref{Thm:1.3.1} is equivalent to
		\begin{equation}\label{super:1}
		\|\mu|M^{r,\lambda}\|\le \gamma_3, \qquad \lambda := {2r/(p-1)}.
		\end{equation}
		Since $r\in (1,p)$, we can take $q\in(p, \infty)$ so that
		\[
		1<\frac{q}{p} < r <q.
		\]
		Set 
		\[
		Y:= L^\infty((0,1); M^{q,\lambda})
		\]
		and
		\[
		\|u|Y\| := \sup_{t\in(0,1)} t^\beta \|u |M^{q,\lambda}\|, \quad \beta:= \frac{\alpha}{p-1}\left(1-\frac{r}{q}\right) <\frac{\alpha}{p}.
		\]
		From here, for simplicity of notation, we set $|\cdot|_{q} := \|\cdot|M^{q,\lambda}\|$.  For $u\in Y$, define
		\[
		\Phi[u](t):= P_\alpha(t)\mu +\alpha\int_0^t (t-s)^{\alpha-1} S_\alpha(t-s)|u(s)|^p \,ds.
		\]
		Denote $B_M := \{u\in Y;\|u|Y\|\le M\}$ for $M>0$, which will be fixed later, 
          and set 
		\[
		d(u,v) := \|u-v|Y\| \qquad u, v \in B_M.
		\]
		Then $(B_M,d)$ becomes a complete metric space.
		 We are going to show below that $\Phi$ maps $B_M$ into itself for all $t \in (0,1)$.
		Since 
		\[
		-\frac{1}{p-1}\left(1-\frac{r}{q}\right)>-1, \quad 1-\frac{r}{q}>0,
		\]
		by \eqref{eq:1.4} we have
		\[
		\int_0^\infty (1+\theta^{-\frac{1}{p-1}(1-\frac{r}{q})})h_{\alpha}(\theta)\, d\theta= 1+\frac{\Gamma\left(1-\frac{1}{p-1}(1-\frac{r}{q})\right)}{\Gamma\left(1-\frac{\alpha }{p-1}(1-\frac{r}{q})\right)}<\infty
		\]
		and
		\[
		\int_0^\infty (\theta +\theta^{1-\frac{r}{q}})h_{\alpha}(\theta)\, d\theta = \frac{\Gamma(2)}{\Gamma(1+\alpha)}+\frac{\Gamma\left(2-\frac{r}{q}\right)}{\Gamma\left(1+\alpha\left(1-\frac{r}{q}\right)\right)}<\infty.
		\]
	     Note that the constants above are uniformly bounded near $\alpha=1$, and	
           that $|f^p|_{q/p} = |f|_q^p$ for all $f\in M^{q,\lambda}({\bf R}^N)$. 
       Let $u\in B_M$ and $t\in (0,1]$. 
		It follows from Proposition~\ref{MpMq} and \eqref{super:1} that
		\begin{equation*}
		\begin{split}
		t^\beta |P_\alpha (t) \mu |_q
		&\le t^\beta \int_0^\infty h_\alpha(\theta) |e^{t^\alpha \theta \Delta}\mu |_q\, d\theta ds\\
		&\le C |\mu|_r\int_0^\infty h_\alpha(\theta )(1+\theta^{-\frac{1}{p-1}(1-\frac{r}{q})})\, d\theta \\
		&\le C_1\gamma_3
		\end{split}
		\end{equation*}
		for $0<t<1$. Moreover, by Proposition~\ref{MpMq},
		\begin{equation*}
		\begin{split}
		&t^\beta \left| \alpha\int_0^t (t-s)^{\alpha-1} S_\alpha(t-s)|u(s)|^p\,ds \right|_q \\
		&=  \alpha t^\beta \int_0^t (t-s)^{\alpha-1} \int_0^\infty \theta h_\alpha(\theta)|e^{(t-s)^\alpha \theta \Delta}|u(s)|^p|_q  \,d\theta ds\\
		&\le   C  t^\beta \int_0^t (t-s)^{(1-\frac{r}{q})\alpha-1} |u(s)|_q^p  \,ds \int_0^\infty  h_\alpha(\theta) (\theta +\theta^{1-\frac{r}{q}}) \,d\theta \\
		&\le CM^p t^\beta\int_0^t(t-s)^{(1-\frac{r}{q})\alpha-1}s^{-p\beta} \, ds\\
		&\le C_2M^pB\left(\alpha\left(1-\frac{r}{q}\right), 1-\frac{p\alpha}{p-1}\left(1-\frac{r}{q}\right) \right)\\
		\end{split}
		\end{equation*}
		for $0<t<1$. 
		Note that $C_1$ and $C_2$ are independent of $\alpha$.
		For $c_*>0$, set
		\begin{equation}
		\gamma_3 := c_* B\left(\alpha\left(1-\frac{r}{q}\right), 1-\frac{p\alpha}{p-1}\left(1-\frac{r}{q}\right) \right)^{-\frac{1}{p-1}}
		\end{equation}
		and
		\[
		M:= 2C_1c_*B\left(\alpha\left(1-\frac{r}{q}\right), 1-\frac{p\alpha}{p-1}\left(1-\frac{r}{q}\right) \right)^{-\frac{1}{p-1}}.
		\]
		Taking a sufficiently small $c_*>0$ if necessary, we have by the above inequalities
		\begin{equation*}
		\begin{split}
		\|\Phi[u]|Y\|
		&=\sup_{t\in(0,1)}t^\beta |\Phi[u](t)|_q \\
		&\le (C_1c_*+2^pC_1^pC_2c_*^p)B\left(\alpha\left(1-\frac{r}{q}\right), 1-\frac{p\alpha}{p-1}\left(1-\frac{r}{q}\right) \right)^{-\frac{1}{p-1}}\\
		& \le 2C_1c_*B\left(\alpha\left(1-\frac{r}{q}\right), 1-\frac{p\alpha}{p-1}\left(1-\frac{r}{q}\right) \right)^{-\frac{1}{p-1}}=M.
		\end{split}
		\end{equation*}
		This implies that $\Phi (B_M) \subset  B_M$.
		
		Next, we show that $\Phi: B_M \to B_M$ is a contraction mapping. Take any $u,v\in B_M$, then by manipulations similar to those above, 
		 we have
		\begin{equation}
		\label{super:3}
		\begin{split}
		& t^\beta |\Phi[u](t)-\Phi[v](t)|_q \\
		&\le \alpha t^\beta \int_0^t (t-s)^{\alpha-1} \int_0^\infty \theta h_\alpha(\theta) |e^{(t-s)^\alpha \theta \Delta}(|u(s)|^p-|v(s)|^p)|_q  \, d\theta ds\\
		&\le \alpha t^\beta \int_0^t (t-s)^{\alpha-1} \int_0^\infty \theta h_\alpha(\theta) 
		(1+[(t-s)^\alpha \theta]^{-\frac{r}{q}})|(|u(s)|^p-|v(s)|^p)|_\frac{q}{p}  \, d\theta ds\\
		&\le \alpha t^\beta\int_0^t (t-s)^{(1-\frac{r}{q})\alpha-1} |(|u(s)|^p-|v(s)|^p)|_\frac{q}{p}  \,  ds
		\int_0^\infty  h_\alpha(\theta) 
		(\theta+\theta^{1-\frac{r}{q}})\,d\theta \\
		&\le  C t^\beta \int_0^t (t-s)^{(1-\frac{r}{q})\alpha-1}  ||u(s)|^p-|v(s)|^p|_\frac{q}{p} \, ds\\
		&\le  C t^\beta \int_0^t (t-s)^{(1-\frac{r}{q})\alpha-1}  (|u(s)|_q^{p-1} + |v(s)|_q^{p-1})|u(s)-v(s)|_q  \, ds\\
		&\le  C  M^{p-1}t^\beta  \int_0^t (t-s)^{(1-\frac{r}{q})\alpha-1} s^{-p\beta}\, ds \, d(u,v)\\
		&\le C_3 M^{p-1} B\left(\alpha\left(1-\frac{r}{q}\right), 1-\frac{p\alpha}{p-1}\left(1-\frac{r}{q}\right) \right)d(u,v)
		\end{split}
		\end{equation}
		for $0<t<1$. 
		Note that $C_3$ is independent of $\alpha$.
		Taking a sufficiently small $c_*>0$ if necessary, we get by \eqref{super:3}
		\[
		d(\Phi[u], \Phi[v]) \le 2^{p-1}C_1^{p-1}C_3c_*^{p-1} d(u,v) \le \frac{1}{2} d(u,v)
		\] 
		for $u,v\in B_M$. 
		Note that $c_*$ is independent of $\alpha$.
		This implies that $\Phi: B_M \to B_M$ is a contraction mapping. 
		Therefore, by the contraction mapping theorem, we know $\Phi$ has a unique fixed point $u\in B_M$. 
		Obviously, if $\mu$ is a nonnegative measurable function in ${\bf R}^N$, we see that $u(x,t)\ge0$ for almost all $t\in[0,1)$ and $x\in{\bf R}^N$.
		Thus, $u$ is a  solution of problem \eqref{TFF} in $(0,1)\times{\bf R}^N$ 
         and Theorem~\ref{Thm:1.3} is proved.
	\end{proof}

	%%%%%%%%%%%%%%%%%%%%%%%%%%%%%%%%%%%%%
	\section{Collapse of the solvability}\label{section. Collapse of the solvability}
	%%%%%%%%%%%%%%%%%%%%%%%%%%%%%%%%%%%%%
	
	%%%%%%%%%%%%%%%%%%%%%%%%%%%%%%%%%%%%%
	\subsection{Preliminaries.}
	%%%%%%%%%%%%%%%%%%%%%%%%%%%%%%%%%%%%%
	
	First, we define the life span of solutions of problem \eqref{TFF}. For this purpose, we recall the following result given in \cite{GMS22}, which is based on \cite{RobSier2013}. See also \cite{HisaIshige18, HIT18, IshiKawaSier2016, LuchYama19}.
	\begin{proposition}
		\label{Prop:min}
		Let $T\in(0,\infty]$.
		If problem \eqref{TFF} possesses a supersolution in $[0,T)\times{\bf R}^N$, then problem \eqref{TFF} possesses a solution $\overline{u}$ in $[0,T)\times{\bf R}^N$ in the sense of Definition {\rm \ref{Def:1.1}}.
		Furthermore, $\overline{u}$ satisfies
		\begin{equation}
		\label{min}
		0\le \overline{u}(x,t) \le v(x,t) \quad \mbox{for a.a.} \quad t \in (0,T), \, x\in{\bf R}^N
		\end{equation}
		for all solutions $v$ of  problem \eqref{TFF} in $[0,T)\times{\bf R}^N$.
	\end{proposition}

	We call the solution $\overline{u}$ satisfying \eqref{min} the minimal solution of problem \eqref{TFF}.
	Since the minimal solution is unique, we can define the life span of solutions of problem \eqref{TFF} as
	\[
	T_\alpha\left[ \mu \right]:=\sup\{T'>0 ; \mbox{the minimal solution of \eqref{TFF} exists in} \,\, [0,T')\times{\bf R}^N\}.
	\]
	By virtue of Proposition \ref{Prop:min}, we can regard the solutions which were obtained in Theorems~\ref{Thm:1.2} and \ref{Thm:1.3} as the minimal solutions. 
	%%%%%%%%%%%%%%%%%%%%%%%%%%%%%%%%%%%%%
	\subsection{Collapse of the global-in-time solvability.}
	%%%%%%%%%%%%%%%%%%%%%%%%%%%%%%%%%%%%%
	In this subsection, we prove Theorem~\ref{Thm:4.1} which indicates the occurrence mechanism of the collapse of the global-in-time solvability. 
       As introduced in Section~1, let
	\[
	\mathcal{G}_{\alpha}:=\left\{ ~\! \nu \in \mathcal{M} ~\! ; ~\! 
         \text{\eqref{TFF} with } u(0) = \nu \geq 0 \ \text{ possesses a global-in-time solution.}
	\right\}.
	\]
	In other words, 
	\[
	\mathcal{G}_{\alpha}= \left\{ ~\! \nu\in \mathcal{M} ~\!; ~\!
	\nu \geq 0, \  T_{\alpha}[\nu]=\infty ~\! \right\}.
	\]
                        
	\begin{proof}[Proof of Theorem~\ref{Thm:4.1}] 
	Let $\alpha\in[1/2,1)$.
		The inclusion
		\[
		\mathcal{B}^+\left( C_1(1-\alpha)^{\frac{N}{2}} \right) \subset \mathcal{G}_{\alpha}
		\]
		is easily deduced from Theorem~\ref{Thm:1.2}. On the other hand, suppose that $\mu \in \mathcal{G}_\alpha$. Then, problem \eqref{TFF} possesses a global-in-time solution and
		\eqref{Thm:1.1.2} holds for all $T>1$.
		Letting $\sigma=T^{\alpha/4} (<T^{\alpha/2})$, we have
		\begin{equation*}
		\sup_{z\in{\bf R}^N} \mu(B(z;T^\frac{\alpha}{4})) \le \gamma'_1 \left(\int_{(16\sqrt{T})^{-1}}^{1/4} t^{-\alpha} \, dt \right)^{-\frac{N}{2}}.
		\end{equation*}
		Since $\limsup_{\alpha\to1-0} \gamma'_1(\alpha) \in (0,\infty)$,
		we have
		\[
		\overline{\gamma} := \sup_{\alpha\in[1/2,1)} \gamma'_1(\alpha) \in (0,\infty).
		\]
		Now taking $T\to\infty$, we obtain
		\begin{equation*}
		\mu({\bf R}^N) \le \gamma'_1\left(\int_{0}^{1/4} t^{-\alpha} \, dt \right)^{-\frac{N}{2}}= 4^{\alpha-1} \gamma'_1 (1-\alpha)^\frac{N}{2}\le \overline{\gamma}(1-\alpha)^\frac{N}{2}.
		\end{equation*}
		This implies that $\mu\in\mathcal{B}^+(C_2 (1-\alpha)^{N/2})
		$ with $C_2= \overline{\gamma}$.
		Then \eqref{Thm:4.1.1} follows.
	\end{proof}
	
	%%%%%%%%%%%%%%%%%%%%%%%%%%%%%%%%%%%%%%%%%%%%%%%%%%%%%
	\subsection{Collapse of the local-in-time solvability}
	%%%%%%%%%%%%%%%%%%%%%%%%%%%%%%%%%%%%%%%%%%%%%%%%%%%%%
	In this subsection, we prove Theorem~\ref{Thm:4.2}. As introduced in Section~1, for $0<\epsilon<N/2$, let
	\[
	f_\epsilon(x) := |x|^{-N}\left|\log|x|\right|^{-\frac{N}{2}-1+\epsilon}\chi_{B(0;e^{-{3/2}})}.
	\]
	Note that $f_{\epsilon}\in L^1({\bf R}^N)$ and $f_{\epsilon}$ is monotonically decreasing with respect to $|x|$. 
	It is observed that problem \eqref{Fjt} with $\mu = f_{\epsilon}$ possesses no local-in-time solutions. In contrast, the local-in-time solvability of problem \eqref{TFF} for any $L^1({\bf R}^N)$ initial data is guaranteed. The occurrence mechanism of this collapse of the local-in-time solvability can be explained by establishing the precise $\alpha$-dependence of the life span for $f_{\epsilon}$.
	
	\begin{proof}[Proof of Theorem~\ref{Thm:4.2}]		
		It follows that
		\begin{equation*}
		\sup_{z\in{\bf R}^N} \int_{B(z;\sigma)} f_{\epsilon}(x)\,dx = \int_{B(0;\sigma)} f_{\epsilon}(x)\,dx = C \left( \log \frac{1}{\sigma} \right)^{-\frac{N}{2}+\epsilon}
		\end{equation*}
		for all $\sigma\le e^{-{3/2}}$. 
		Firstly, we derive an upper estimate. Define 
		\begin{equation}
		\label{eq:overT}
		\overline{T_\alpha} [\kappa f_\epsilon] := \min\left\{ T_{\alpha}[\kappa f_{\epsilon}], e^{-\frac{3}{\alpha}}\right\} \le  T_{\alpha}[\kappa f_{\epsilon}].
		\end{equation}
		We can assume that problem \eqref{TFF} with $\mu=\kappa f_\epsilon$ possesses a solution in $[0,\overline{T_\alpha}[\kappa f_\epsilon])\times{\bf R}^N$.
		Since $(1-\alpha)^{{1/(1-\alpha)}}<1$, we can use Theorem~\ref{Thm:1.1} with $T=\overline{T_{\alpha}}[\kappa f_{\beta}]$ and
		\[
		\sigma = \overline{T_\alpha} [\kappa f_\epsilon]^\frac{\alpha}{2}  (1-\alpha)^{\frac{\alpha}{2(1-\alpha)}} \le \overline{T_{\alpha}}[\kappa f_{\epsilon}]^\frac{\alpha}{2} \le e^{-\frac{3}{2}}.
		\]
		Therefore, the necessary condition \eqref{Thm:1.1.2} provides
		\begin{equation*}
		\kappa \left( \log \frac{1}{\sigma} \right)^{-\frac{N}{2}+\epsilon}
		\le C \gamma_1' (1-\alpha)^{\frac{N}{2}} \left( \left( \frac{1}{4} \right)^{1-\alpha}-\left( \frac{1}{16} \right)^{1-\alpha}(1-\alpha) \right)^{-\frac{N}{2}}.
		\end{equation*}
		Note that if $\alpha$ approaches $1$, we can take $C>0$ independent of $\alpha$ such that
		\[
		\left( \left( \frac{1}{4} \right)^{1-\alpha}-\left( \frac{1}{16} \right)^{1-\alpha}(1-\alpha) \right)^{-\frac{N}{2}}<C.
		\]
		Therefore, combining with $\limsup_{\alpha\to1-0} \gamma'_1(\alpha)\in (0,\infty)$, we get
		\[
		\kappa \left( \log \frac{1}{\sigma} \right)^{-\frac{N}{2}+\epsilon} \le C(1-\alpha)^{\frac{N}{2}},
		\]
		where $C>0$ is independent of $\alpha$, near $\alpha =1$. This implies that
		\[
		\sigma \le \exp \left( -C \kappa^{\frac{2}{N-2\epsilon}} (1-\alpha)^{-\frac{N}{N-2\epsilon}} \right),
		\]
		that is 
		\begin{equation*}
		\begin{split}
		\overline{T_{\alpha}}[\kappa f_{\epsilon}]
		&\le (1-\alpha)^{-\frac{1}{1-\alpha}} \exp \left( -\frac{2}{\alpha}C \kappa^{\frac{2}{N-2\epsilon}} (1-\alpha)^{-\frac{N}{N-2\epsilon}} \right)\\
		&\le (1-\alpha)^{-\frac{1}{1-\alpha}} \exp \left( -C \kappa^{\frac{2}{N-2\epsilon}} (1-\alpha)^{-\frac{N}{N-2\epsilon}} \right)
		\end{split}
		\end{equation*}
		near $\alpha=1$. Here, we claim that
		\[
		(1-\alpha)^{-\frac{1}{1-\alpha}} \exp\left( -C \kappa^{\frac{2}{N-2\epsilon}} (1-\alpha)^{-\frac{N}{N-2\epsilon}} \right) \le \exp \left( -C' \kappa^{\frac{2}{N-2\epsilon}} (1-\alpha)^{-\frac{N}{N-2\epsilon}} \right)
		\]
		for appropriate $C'>0$. Indeed, it suffices to show that for any $K>0$ and $\beta>1$,
		\begin{equation}\label{eq. inequality}
		(1-\alpha)^{-\frac{1}{1-\alpha}} \exp \left( -K (1-\alpha)^{-\beta} \right)<1
		\end{equation}
		holds near $\alpha=1$, which is assured by Lemma~\ref{lem. inequality} below. 
		Therefore, it follows from \eqref{eq:overT} that
		\[
		\overline{T_\alpha} [\kappa f_\epsilon] 
		 \le \exp \left( -C \kappa^{\frac{2}{N-2\epsilon}} (1-\alpha)^{-\frac{N}{N-2\epsilon}} \right)
		\]
		near $\alpha =1$.
		Since the right-hand side of the above inequality tends to $0$ as $\alpha \to 1-0$,
		there exists  $\alpha(\kappa,\epsilon)\in (0,1)$ such that all the above calculations are justified and
		\[
		e^{-\frac{3}{\alpha}} \ge \exp \left( -C \kappa^{\frac{2}{N-2\epsilon}} (1-\alpha)^{-\frac{N}{N-2\epsilon}} \right)
		\]
		holds for $\alpha(\kappa,\epsilon ) <\alpha <1$.
		This implies that
		\[
		T_{\alpha}[\kappa f_{\epsilon}] \le \exp \left( -C \kappa^{\frac{2}{N-2\epsilon}} (1-\alpha)^{-\frac{N}{N-2\epsilon}} \right)
		\]
		for $\alpha(\kappa,\epsilon ) <\alpha <1$,
		and the upper estimate of $T_{\alpha}[\kappa f_{\epsilon}] $ has been obtained.
		
		Let us deduce the lower estimate. 
		Define
		\[
		\hat{T}_\alpha := \exp \left( -C \kappa^\frac{2}{N-2\epsilon}(1-\alpha)^{-\frac{N}{N-2\epsilon}} \right)
		\]
		for appropriate $C>0$.
		It suffices to show that problem \eqref{TFF} possesses a solution in $[0,\hat{T}_\alpha)\times{\bf R}^N$.
		Since $\lim_{\alpha\to1-0} \hat{T}_\alpha=0$, we  may assume that $\alpha$ is close enough to $1$ so that $\hat{T}_\alpha<e^{-3/2}$.
		We have
		\begin{equation*}
		\begin{split}
		\sup_{z\in{\bf R}^N} \int_{B(z;\hat{T}_\alpha^{\alpha/2})}\kappa f_\epsilon(x)\,dx
		&= \int_{B(0;\hat{T}_\alpha^{\alpha/2})}\kappa f_\epsilon(x)\,dx\\
		& = \kappa \left( \log \frac{1}{\hat{T}_\alpha^{{\alpha/2}}} \right)^{-\frac{N}{2}+\epsilon}\\
		& = C (1-\alpha)^\frac{N}{2} \to 0\quad \mbox{as} \quad \alpha\to1-0.
		\end{split}
		\end{equation*}
		This implies that \eqref{Thm:1.2.1} holds near $\alpha=1$ and guarantees the solvability on $[0,\hat{T}_\alpha)$.
		In particular, $\hat{T}_\alpha \le T_{\alpha}[\kappa f_{\epsilon}]$ is verified
		and  the lower estimate of $T_{\alpha}[\kappa f_{\epsilon}] $ has been obtained.
	\end{proof}

%%%%%%%%%%%%%%%%%%%%%%%%%%%%%%%%%%%%%%%%%%%%%%%%%%%%%%%%%%%%%%%%%%%%%%%%%%%%%%%%%%%%%%%%%%%%%%%%%%%
%%%%%%%%% Lemma 4.1     %%%%%%%%%%%%%%%%%%%%%%%%%%%%%%%%%%%%%%%%%%%%%%%%%%%%%%%%%%%%%%%%%%%%%%%%%%%
%%%%%%%%%%%%%%%%%%%%%%%%%%%%%%%%%%%%%%%%%%%%%%%%%%%%%%%%%%%%%%%%%%%%%%%%%%%%%%%%%%%%%%%%%%%%%%%%%%%	
	\begin{Lemma}\label{lem. inequality}
		Inequality \eqref{eq. inequality} holds true. More precisely, for all $K>0$ and $\beta>1$,
		\[
		\tau^{-\frac{1}{\tau}} \exp\left( -K \tau^{-\beta} \right)<1
		\]
		for sufficiently small $\tau>0$. 
	\end{Lemma}
%%%%%%%%%%%%%%%%%%%%%%%%%%%%%%%%%%%%%%%%%%%%%%%%%%%%%%%%%%%%%%%%%%%%%%%%%%%%%%%%%%%%%%%%%%%%%%%%%%%
	\begin{proof}
		The conclusion is equivalent to 
		\[
		\tau^{-1}\exp\left(- K \tau^{1-\beta} \right) < 1
		\]
		near $\tau=0$ and it is true since $ \tau^{-1}\exp\left(- K \tau^{1-\beta}\right) \to 0$ as $\tau\to 0$.
	\end{proof}
	
	\begin{Remark}
		{\rm 
		In Theorem~\ref{Thm:4.2}, the lower estimate is true for all $0<\alpha<1$, while the upper estimate does not hold for any $0<\alpha<1$ in general. Indeed, for each $\alpha$, if $\kappa$ is sufficiently small, then $T_{\alpha}[\kappa f_{\epsilon}]=\infty$ (See Theorem~\ref{Thm:4.1}). However, no matter how small $\kappa$ is, $T_{\alpha}[\kappa f_{\epsilon}]$ becomes finite if $\alpha$ approaches $1$ (collapse of the \textit{global-in-time} solvability), and $\alpha \to 1-0$ yields $T_{\alpha}[\kappa f_{\epsilon}]\to 0$ (collapse of the \textit{local-in-time} solvability).
	}
	\end{Remark}

%%%%%%%%%%%%%%%%%%%%%%%%%%%%%%%%%%%%%
\section*{Conclusion and Acknowledgments.}
%%%%%%%%%%%%%%%%%%%%%%%%%%%%%%%%%%%%%
In the present paper, we aim to clarify the occurrence mechanism of
the collapse of the global and local-in-time solvability of \eqref{TFF} with $p=p_F$ when $\alpha$ approaches to $1-0$. To carry out this program, we derive the necessary and sufficient conditions for the solvability of \eqref{TFF} in a precise way, and apply these results to reveal the collapse mechanisms. Through this procedure, we conclude that the collapse of the global-in-time solvability results from the shrinkage of $\mathcal{G}_{\alpha}$, the set of initial data admitting global-in-time solutions, towards $\{0\}$ as $\alpha \to 1$. On the other hand, the collapse of the local-in-time solvability of \eqref{TFF} for the special family $\{f_{\epsilon}\}$ of initial data is caused by the decay of the life span $T_{\alpha}\left[ f_{\epsilon} \right]$ of the solution $u(t)$ of \eqref{TFF} with $u(0)= f_{\epsilon}$ towards zero as $\alpha \to 1$.

Finally, the authors of this paper would like to thank the referees for carefully reading the manuscript and relevant remarks. The authors also wishes to thank Professor Mitsuharu \^Otani for his helpful advice to the preparation of the paper.
	
	%%%%%%%%%%%%%%%%%%%%%%%%%%%%%%%%%%%%%%
	%%%%%%%%%%%%    references    %%%%%%%%%%%%%%%%%%
	%%%%%%%%%%%%%%%%%%%%%%%%%%%%%%%%%%%%%%
	\bibliographystyle{plain}
	\bibliography{ref}

\end{document}